\documentclass[12pt]{article}
\usepackage{amsmath, amsfonts, amsthm, amssymb, color, hyperref, enumerate, extarrows, cite}

 \usepackage[titletoc]{appendix}

\newtheorem{theorem}{Theorem}[section]
\newtheorem{definition}[theorem]{Definition}
\newtheorem{corollary}[theorem]{Corollary}
\newtheorem{lemma}[theorem]{Lemma}
\newtheorem*{notation*}{Notation}
{\theoremstyle{definition}
}
 
\newtheorem{proposition}[theorem]{Proposition}

\numberwithin{equation}{section}

{\theoremstyle{definition}
\newtheorem{remark}[theorem]{Remark}}
\usepackage[T1]{fontenc}

\hypersetup{
	colorlinks   = true,
	citecolor    = magenta}

\usepackage[titletoc]{appendix}

\begin{document}

\title{\vspace{-1.2cm} \bf Rigidity theorems by capacities and kernels\rm}

\author{Robert Xin Dong, \quad John N. Treuer \quad and\quad Yuan Zhang}
\date{}

\maketitle

\begin{abstract}

For any open hyperbolic Riemann surface $X$, the Bergman kernel $K$, the logarithmic capacity $c_{\beta}$, and the analytic capacity $c_{B}$ satisfy the inequality chain $\pi K \geq c^2_{\beta} \geq c^2_B$; moreover, equality holds at a single point between any two of the three quantities if and only if  $X$ is biholomorphic to a disk possibly less a relatively closed polar set.
We extend the inequality chain by showing that $c_{B}^2 \geq \pi v^{-1}(X)$ on planar domains, where $v(\cdot)$ is the Euclidean volume, and characterize the extremal cases when equality holds  at one point. Similar rigidity theorems concerning the Szeg\"{o} kernel, the higher-order Bergman kernels, and the sublevel sets of the Green's function are also developed. Additionally, we explore rigidity phenomena related to the multi-dimensional Suita conjecture.

\end{abstract}

\renewcommand{\thefootnote}{\fnsymbol{footnote}}
\footnotetext{\hspace*{-7mm} 
\begin{tabular}{@{}r@{}p{16.5cm}@{}}
& Keywords. analytic capacity, Bergman kernel, Green's function, logarithmic capacity, open Riemann surface, Suita conjecture, Szeg\"{o} kernel\\
& Mathematics Subject Classification. Primary 32A25; Secondary 30C40, 31A15, 30C85, 32D15
\end{tabular}}

\section{Introduction}

An open Riemann surface $X$ is said to be (potential-theoretically) hyperbolic if it admits 
a non-constant negative subharmonic function,
and parabolic if it does not. 
Consider the on-diagonal Bergman kernel $K(z)|dz|^2$, the logarithmic capacity $c_{\beta}(z)|dz|$, and the analytic capacity $c_B(z)|dz|$ on $X$, where $z$ is some local coordinate. 
These three quantities are invariant under changes of local coordinates.
 In 1972, Suita determined a simple relationship between $K$ and $c_B$.

\begin{theorem}[Suita's Theorem in \normalfont \cite{S}] \label{Suita Theorem}
Suppose $X$ is an open hyperbolic Riemann surface.  Then for $z\in X$,
$$\pi K(z) \geq c_B^2(z)$$ and equality holds at some $z_0 \in X$ if and only if $X$ is biholomorphically equivalent to the unit disk less a (possibly empty) closed set of inner capacity zero.
\end{theorem}

A closed set of inner capacity zero is a relatively closed polar set. In the same paper, Suita conjectured that  $c_\beta$ would satisfy a similar inequality with rigidity as follows.
\medskip

\noindent\textbf{Suita Conjecture\normalfont\ \cite{S}:} \label{Suita conjecture}
{\it Suppose $X$ is an open hyperbolic Riemann surface.  Then for $z\in X$, $$
\pi K(z) \geq c_\beta^2(z)
$$ and equality holds at some $z_0 \in X$ if and only if  $X$ is biholomorphically equivalent to the unit disk less a (possibly empty) closed set of inner capacity zero.} 

\medskip

Towards the resolution of the Suita Conjecture,  Ohsawa  first demonstrated in \cite{O95} that the Suita Conjecture was connected to his Ohsawa-Takegoshi $L^2$ extension theorem, and he proved that $750\pi K \geq c_{\beta}^2$.  By further considering the sharp $L^2$ extension problem, B\l{}ocki \cite{Bl} established the optimal inequality $\pi K \geq c_{\beta}^2 $ for bounded domains in $\mathbb{C}$.  Later, Guan and Zhou \cite{GZ3} proved both the inequality and equality parts of the conjecture for open Riemann surfaces. See also \cite{BL} for a variational approach by Berndtsson and Lempert. 
Suita's theorem and the resolution of the conjecture may be called rigidity theorems as equality at one point between $\pi K$ and either  $c_{\beta}$ or $c_B$ determines the surface 
up to biholomorphism. 

\medskip

It is straightforward to show $c_{\beta}^2 \geq c_B^2$, so by the works of B\l{}ocki, and Guan and Zhou, for any open hyperbolic Riemann surface it holds that
\begin{equation} \label{Inequality of surface quantities}
\pi K \geq c_{\beta}^2 \geq c_B^2.
\end{equation}
The characterization of the equality part for the inequality $c_{\beta} \geq c_B$ can be deduced from a result of Minda \cite{Min} on the behavior of $c_{\beta}(z)|dz|$ under holomorphic mappings.  His result 
used the sharp form of the Lindel\"of principle.  In Appendix \ref{AppendixA}, we provide a more direct proof of the equality characterization without adopting this principle.

\medskip

We first restrict our attention to domains in $\mathbb{C}$.
Throughout this paper, denote a disk by $
\mathbb{D}(z_0, r): = \{z \in \mathbb{C}: |z - z_0| < r\}$ and let $  \mathbb{D} := \mathbb{D}(0, 1).
$ 
For the Euclidean volume $v(\cdot)$, we use the convention that $v^{-1} (\Omega)  = 0$ when $v(\Omega) = \infty$. Following Ahlfors \cite{AB50}, a compact set $E \subset \mathbb{C}_{\infty} = \mathbb{C} \cup \{\infty\}$ is said to be {\it a null set of class $\mathcal{N}_{B}$} if all bounded holomorphic functions on $\mathbb{C}_{\infty} \setminus E$ are constants.  Compact polar sets are always of class $\mathcal{N}_B$.

\begin{theorem}[Rigidity theorem of $c_B$ and $c_{\beta}$]\label{rigidity of analytic and logarithmic capacity with reciprocal of volume}
Let $\Omega \subset \mathbb{C}$ be a domain.  Then
\begin{equation} \label{analytic capacity theorem main eqn}
c_B^2(z) \geq {\pi \over v(\Omega)}, \quad z \in \Omega.
\end{equation}
  Moreover, equality holds at some
$z_0 \in \Omega$ if and only if either of the following holds true:
\begin{enumerate}

\item [$1.$] $v(\Omega) < \infty$ and $\Omega = {\mathbb D}( z_0, \sqrt{\pi^{-1}v(\Omega)}) \setminus P$, where  $P$ satisfies
$$
P \cap \overline{\mathbb{D}(z_0, s)} \in \mathcal{N}_B, \quad \hbox{for all } s < \sqrt{\pi^{-1}v(\Omega)};
$$
if additionally $$c_{\beta}^2(z_0) = {\pi \over v(\Omega)},$$ then $P$ above is a relatively closed polar set;

\item [$2.$] $v(\Omega) = \infty$ and $\Omega =\mathbb{C}\setminus P$ where $P \in \mathcal{N}_B$.

\end{enumerate}

\end{theorem}

 It is worth pointing out that,   as demonstrated in Remark \ref{EX}, the equality condition in \eqref{analytic capacity theorem main eqn} does not necessarily imply that $P$ is polar.  See also Theorem \ref{jth} and Lemma \ref{Volume Lemma} for rigidity theorems with weaker assumptions than \eqref{analytic capacity theorem main eqn}.

\medskip

Theorem \ref{rigidity of analytic and logarithmic capacity with reciprocal of volume} extends  the chain of inequalities \eqref{Inequality of surface quantities} to
\begin{equation} \label{Domain Inequality Chain}
  \pi K \geq c_{\beta}^2 \geq  c_B^2 \geq  \frac{\pi}{v(\Omega)},
\end{equation}
and essentially gives the equality conditions between $\pi v^{-1}(\Omega)$ and the other quantities in \eqref{Domain Inequality Chain}  for a hyperbolic domain $\Omega\subset \mathbb C$.
As a corollary (see Corollary \ref{theorem 1}), we can reprove the main theorem of \cite{DT}, a rigidity theorem which established the equality conditions of $K(z) \geq v^{-1}(\Omega)$. Moreover, let $K^{(j)}, j \in \mathbb{N}$, denote the higher-order Bergman kernels of a domain $\Omega \subset \mathbb{C}$. Combining  Theorem \ref{rigidity of analytic and logarithmic capacity with reciprocal of volume} with a result of B\l{}ocki and Zwonek in  \cite{BlZw}, we get the inequalities  
$$
K^{(j)}(z) \geq \frac{j!(j+1)!  \pi^j}{ v^{j + 1}(\Omega)},  \quad  z \in \Omega, \quad j \in \mathbb{N},
$$
and the following characterization of their equality conditions.

\begin{corollary}[Rigidity theorem of higher-order Bergman kernels]\label{Higher Bergman Kernel Corollary}
Let $\Omega$ be a domain in $ \mathbb{C}$. Then for each $j \in \mathbb{N}$, there exists a $z_0 \in \Omega$ such that
$$
K^{(j)}(z_0) =\frac{j!(j+1)!  \pi^j}{ v^{j + 1}(\Omega) }
$$
if and only if either of the following holds true:

 \begin{enumerate}

\item [$1.$] $v(\Omega) < \infty$ and $\Omega = \mathbb{D}(z_0, \sqrt{\pi^{-1} v(\Omega)}) \setminus P$, where $P$ is a relatively closed polar set;

\item [$2.$]  $v(\Omega) = \infty$ and $\Omega = \mathbb{C} \setminus P$, where $P$ is a closed polar set.
\end{enumerate}

\end{corollary}

A classical question (\cite[p. 20]{Stein}) raised by Stein is {\it what are the relations between the Bergman kernel $K$ and the Szeg\"{o} kernel $S$}? On any bounded domain $\Omega$ in $\mathbb{C}$ with Lipschitz boundary (see \cite{Be16, G72} for the $C^\infty$ smooth case and Proposition \ref{Relation with analytic capacity and szego kernel} for the Lipschitz case),
the Szeg\"{o} kernel $S$ and the analytic capacity $c_B$ satisfy the identity  $c_B(z) = 2\pi S(z)$. Using this and \eqref{Domain Inequality Chain}, we deduce the following inequality chain, which gives a relation between the on-diagonal Bergman and Szeg\"o kernels:

\begin{equation} \label{Lip Domain Inequality Chain}
\sqrt{ \pi K(z)}  \geq c_{\beta} (z) \geq   2\pi  S(z)    \geq  \sqrt{\frac{\pi}{v(\Omega)}} \geq \frac{2\pi}{\sigma(\partial \Omega)}, \quad z \in \Omega.
\end{equation}
Here $\sigma(\partial\Omega)$ denotes the arc-length of $\partial\Omega$, and the last inequality in \eqref{Lip Domain Inequality Chain} is the isoperimetric inequality.  In particular,  the inequalities say that $K$ always dominates $4 \pi  S^2$  on Lipschitz domains.  See also  \cite{CF} and the references therein for results on the comparison of these two kernels.  We characterize the equality conditions in \eqref{Lip Domain Inequality Chain} as below.

\begin{theorem} [Rigidity theorem of the Szeg\"{o} kernel] \label{NEW} 
Suppose $\Omega \subset \mathbb{C}$ is a bounded domain with Lipschitz boundary. Then, 
 \begin{enumerate}
\item [$1.$]\label{Case 1 NEW}  there exists some $z_0\in \Omega$ such that 
$$
 2\pi  S(z_0) =  \quad c_{\beta} (z_0)  \quad \text{or} \quad \sqrt{ \pi K(z_0)} 
 $$
   if and only if $\Omega$ is simply connected; 
 \item [$2.$]\label{Case 2 NEW}  there exists some $z_0\in \Omega$ such that 
 $$
 2\pi  S(z_0) =  \quad \sqrt{\frac{\pi}{v(\Omega)}}   \quad \text{or} \quad    \frac{2\pi}{\sigma(\partial \Omega)}\quad  \text{or} \quad \frac{1 } {\delta  (z_0)}% \quad \text{or} \quad \frac{1 } {\delta  (z_0)} 
 $$
  if and only if $\Omega = \mathbb{D}(z_0, \sqrt{\pi^{-1} v(\Omega) })$.
  \end{enumerate}
   
\end{theorem}

Next, by a monotonicity result of B\l{}ocki and Zwonek in \cite{BZ15} concerning the Euclidean volume of the sublevel sets of the Green's function, for $ z\in \Omega, -\infty < t_1 < t_2 < 0$,

\begin{equation} \label{inequalities with log capacity and sublevel sets of green's function}
 c_{\beta}^2(z) \ge \frac{\pi e^{2{t_1}} }{v{(\{G( \cdot, z) < t_1\})}}   \ge \frac{\pi e^{2{t_2}} }{v{(\{G( \cdot, z) < t_2\})}} \ge \, \frac{\pi}{ v(\Omega)}.
\end{equation}
See Remark \ref{mon}. Using the isoperimetric inequality and the classical PDE theory on the unique continuation property for harmonic functions, we obtain the following  rigidity theorem which examines the extremal cases  and characterizes domains on which equalities in \eqref{inequalities with log capacity and sublevel sets of green's function} hold.

\begin{theorem}[Rigidity theorem of sublevel sets of  Green's function] \label{lc}
Let $\Omega$ be a bounded domain in $\mathbb C$. 
Then $\Omega$ is a disk centered at $z_0$ possibly less a relatively closed polar subset if and only if
either of the following 
holds true:
  \begin{enumerate}
\item [$1.$]  $$c_{\beta}^2(z_0) = \frac{\pi e^{2{t_0}} }{v{(\{G( \cdot, z_0) < t_0\})}}$$
for some  $z_0\in \Omega$, and $t_0\in (-\infty, 0)$;

\item [$2.$] $$\frac{\pi e^{2{t_1}} }{v{(\{G( \cdot, z_0) < t_1\})}}    = \frac{\pi e^{2{t_2}} }{v{(\{G( \cdot, z_0) < t_2\})}} $$
for some  $z_0\in \Omega$, and $t_1\ne t_2$ in $(-\infty, 0)$;

\item [$3.$] $$\frac{\pi e^{2{t_0}} }{v{(\{G( \cdot, z_0) < t_0\})}} = \, \frac{\pi}{ v(\Omega)}$$
for some  $z_0\in \Omega$, and $t_0\in (-\infty, 0)$.
\end{enumerate}
  
\end{theorem}

At last, we study rigidity properties for bounded  domains in $\mathbb{C}^n, n\ge 1$. B\l{}ocki and Zwonek \cite{BZ15}  proved the following \textbf{multi-dimensional Suita Conjecture} concerning the Bergman kernel and the Azukawa indicatrix:  {\it for any pseudoconvex domain $\Omega \subset \mathbb{C}^n$, 
 $$
 K(z)\ge  {{v^{-1} (I^A (z)) } }, \quad z \in \Omega. 
 $$}Here $I^A$ denotes the Azukawa indicatrix (see (\ref{125}) for its definition). As the proof relied on an approximation of the domain by hyperconvex sub-domains, the pseudoconvexity was needed in \cite{BZ15}.
For our final result, by considering the connection between the two involved quantities in the multi-dimensional Suita Conjecture and the Euclidean distance function $\delta(z)$, 
 we derive  the following rigidity theorem without requiring the domains in $\mathbb C^n$ to be  pseudoconvex when $n > 1$.

\begin{theorem}  \label{delta-c}
Let $\Omega$ be a  bounded domain  in $\mathbb C^n$, $n\ge 1$. Then for all $z \in   \Omega$,

\begin{equation}  \label{A} 
v(I^A (z)) \ge \frac{\pi^n }{n! }\delta^{2n}(z),
\end{equation}
 and
\begin{equation}  \label{B} 
K(z)\le \frac{n! }{\pi^n\delta^{2n}(z)}.
\end{equation}
 Moreover, 
 either equality holds at some $z_0\in \Omega$ if and only if $\Omega$ is a ball centered at $z_0$.

\end{theorem}

As a consequence of \eqref{B},  we may add the distance function $\delta$ to the inequality chains \eqref{Domain Inequality Chain},
\eqref{Lip Domain Inequality Chain}, and \eqref{inequalities with log capacity and sublevel sets of green's function}  on bounded domains in $\mathbb C$.  See Remark \ref{Final remark} for more details.

 \medskip

The organization of the paper is as follows. 
In Section \ref{Section four} we prove a rigidity theorem of $c_B$ and $c_{\beta}$, and as applications, in Section \ref{Section five} we prove  rigidity theorems of the Bergman kernels, and new results on the Szeg\"o kernel on Lipschitz domains.  In Section \ref{Section 6}, we investigate rigidity phenomena for the logarithmic capacity and the sublevel sets of the Green's function. 
 In Section \ref{Section 7}, for bounded domains in $\mathbb{C}^n$ we prove Theorem \ref{delta-c}. In the appendices, we prove a rigidity phenomenon between  $c_B$ and $c_{\beta}$ based on a rigidity lemma of the Green's function, and   a stability result of the Szeg\"o kernel on Lipschitz domains.

\section{Rigidity Theorem of $c_B$ and $c_{\beta}$} \label{Section four}
Let $SH^{-}(X)$ denote the set of negative subharmonic functions on an open Riemann surface $X$. Given   $z_0 \in X$,  let $w$ be  a fixed local coordinate  in a neighbourhood of $z_0$ such that $w(z_0) =0$. The (negative) Green's function   is
\begin{equation} \label{Other Green's Function}
G(z, z_0) = \sup\{u(z): u \in SH^{-}(X), \limsup_{z \to z_0} u(z) - \log|w(z)| < \infty\}.
\end{equation}
 An open Riemann surface admits a Green's function if and only if there is a non-constant, negative subharmonic function defined on it.  The Green's function is strictly negative on the surface and harmonic except on the diagonal.

 \medskip

We say that a Borel set $P$ is polar if there is a subharmonic function $u \not\equiv -\infty$ defined in a neighbourhood of $P$ so that $P \subset \{z: u(z) = -\infty\}$.  For a domain $\Omega \subset \mathbb{C}$, a point $\zeta_0 \in \partial \Omega$ is said to be irregular if there is a $\zeta \in \Omega$ such that $\lim_{z \to \zeta_0} G(z, \zeta) \neq 0$ and regular otherwise.  By Kellogg's Theorem, cf. \cite[Theorem 4.2.5]{R95}, the set of irregular boundary points of a domain is a polar set.
The Green's function  on a Riemann surface induces the logarithmic capacity, which is used prominently (see \cite{R95} for its applications) in potential theory. Ahlfors introduced the analytic capacity for domains \cite{A47, AB50} in order to study Painlev\'{e}'s question: \textit{which compact sets $E$ in the complex plane are removable for the bounded holomorphic functions}?

\begin{definition}\label{Logarithmic capacity revisited}
Let $X$ be an open hyperbolic Riemann surface and $z_0\in X$.  The logarithmic capacity $c_{\beta}$ is defined as
 $$
c_{\beta}(z_0) = \lim_{z \to z_0} \exp(G(z, z_0) - \log|w(z)|).
 $$
The analytic capacity of a Riemann surface $X$ is defined as
\begin{equation}\label{analytic capacity}
c_B(z_0) = \sup\left \{ \left|{\partial f \over \partial w}(z_0)\right|: f \in \hbox{Hol}(X, \mathbb{D}), f(z_0) = 0\right \}.
\end{equation}
\end{definition}

  When we wish to emphasize the surface $X$, 
we will use the notations  $G_X$ and $c_{m; X}(\cdot)$, $m = \beta  \text{ or } B$.  If $h:X_1 \to X_2$ is a biholomorphism, then for $z\in X_1$,
 $$
c_{m; X_1}(z) = |h'(z)|c_{m; X_2}(h(z)), \quad m = \beta  \text{ or } B.
 $$
Thus, the logarithmic and analytic capacity are independent of the choice of the local coordinates and define conformally-invariant metrics $c_\beta(z)|dz|$ and $c_B(z)|dz|$. For each $z_0 \in X$, there exists an extremal function $f_0$ for the analytic capacity, unique up to rotation; that is, $f_0$ belongs to the family described in \eqref{analytic capacity} and $\left | {df_0 \over dw}(z_0) \right | = c_B(z_0)$.

\medskip

We restrict our attention  to domains in $\mathbb{C}$.  With this restriction we will be able to examine the relationship between the domain functions $c_\beta, c_B$ and the Euclidean volume of the domain. In the literature, the analytic capacity is often defined in terms of compact sets.

\begin{definition}{\normalfont\cite{G72}} \label{Analytic capacity of compact set}
The analytic capacity $\gamma(E)$ of a compact subset $E \subset \mathbb{C}$ is
$$
\gamma(E) = \sup\{|g'(\infty)|: g \in \hbox{Hol}(\mathbb{C}_{\infty} \setminus E, \mathbb{D}), g(\infty) = 0\}
$$
where
$$
g'(\infty) = \lim_{z \to \infty} z(g(z) - g(\infty)).
$$
\end{definition}

\begin{notation*}
In this section, $\Omega - \{z_0\}$ will denote the translation of the domain $\Omega$ by $z_0$ and
$$
j_{z_0}(z) := {1 \over z - z_0}.
$$
For ease of notation, when $z_0 = 0$, let $j = j_{z_0}$.
\end{notation*}

The two definitions of the analytic capacities presented in this paper are related as follows.

\begin{lemma} \label{Equivalence of definitions lemma} For any domain $\Omega \subset \mathbb{C}$ and $z_0\in \Omega$,
$
c_B(z_0) = \gamma(\mathbb{C}_{\infty} \setminus j_{z_0}(\Omega)).
$
\end{lemma}
\begin{proof} Since $j(\Omega - \{z_0\}) = j_{z_0}(\Omega)$ and
$
c_{B; \Omega}(z_0) = c_{B; \Omega - \{z_0\}}(0),
$
it suffices to prove the lemma when $z_0 = 0 \in \Omega$.
Let $E = \mathbb{C}_{\infty} \setminus j(\Omega)$.
Notice that there is a one-to-one correspondence between the two function sets $\mathcal A: = \{g\in \hbox{Hol}(\mathbb{C}_{\infty} \setminus E, \mathbb{D}): g(\infty) = 0\}$ and  $\mathcal B: =\{h\in \hbox{Hol}(\Omega, \mathbb{D}): h(0) = 0 \}$, as $g \circ j\in \mathcal B$ whenever $g\in \mathcal A$, and vice versa. Moreover, $(g \circ j)'(0)
= g'(\infty)$.  The lemma then follows directly from the definitions of the two analytic capacities.

\end{proof}

The analytic capacity is difficult to compute in general for most domains.  It does however satisfy a lower bound referred to as the {\bf Ahlfors-Beurling Inequality} (see \cite{AB50}, \cite[Theorem 4.6, Chapter III]{G72}). Namely, for any compact set $E \subset \mathbb{C}$,
\begin{equation} \label{Analytic Capacity Lower Bound}
\gamma^2(E) \geq {v(E) \over \pi}.
\end{equation}
 
For any $z\in \Omega$, letting $E$ in (\ref{Analytic Capacity Lower Bound}) be $ \mathbb{C}_{\infty} \setminus j_{z}(\Omega)$, and  combining with Lemma \ref{Equivalence of definitions lemma}, we obtain
\begin{equation} \label{analytic capacity theorem main eqn 2}
c_B^2(z) \geq {v(\mathbb{C}_{\infty} \setminus j_{z}(\Omega)) \over \pi}.
\end{equation}
The following serves as a rigidity theorem concerning  \eqref{analytic capacity theorem main eqn 2}.

\begin{theorem}\label{jth}
Let $\Omega$ be a  domain in $\mathbb C$ with $v(\Omega)<\infty$. There exists a $z_0 \in \Omega$ such that \begin{equation} \label{analytic capacity theorem main eqn 3}
c_B^2(z_0) = {v(\mathbb{C}_{\infty} \setminus j_{z_0}(\Omega)) \over \pi}
\end{equation}
  if and only if  $\Omega = {\mathbb D}( z_1, r) \setminus P$  for some $z_1\in \mathbb C, r>0$, where  $P$ satisfies
\begin{equation}\label{NB}
    P \cap \overline{\mathbb{D}(z_1, s)} \in \mathcal{N}_B, \quad \hbox{for all } s < r.
\end{equation}
If additionally \begin{equation}\label{cbe}
    c_{\beta}^2(z_0) = {v(\mathbb{C}_{\infty} \setminus j_{z_0}(\Omega)) \over \pi},
\end{equation} then $P$ above is a relatively closed polar set.
\end{theorem}

\begin{proof}
If $\Omega = \mathbb{D}(z_1, r)\setminus Q$, where
$$
Q \cap \overline{\mathbb{D}(z_1, s)} \in \mathcal{N}_B, \quad \hbox{for all }s < r,
$$
then  by the classical Schwarz lemma, 
\begin{equation}\label{analytic capcacity of general disk}
c_B(z_0) =  \frac{ r}{r^2-|z_0-z_1|^2}.
\end{equation}
On the other hand, a direct computation gives  
$$
   v(\mathbb C_\infty\setminus j_{z_0}(\mathbb{D}(z_1, r))) = \frac{\pi r^2}{(r^2-|z_0-z_1|^2)^2}.  
$$  
Equation (\ref{analytic capacity theorem main eqn 3}) is proved.

For the other direction, let $E = \mathbb{C}_{\infty}\setminus j_{z_0}(\Omega)$. Then $E$ is compact, and since $v(\Omega)<\infty$,  $v(E)>0$. Define $f \in \mathcal{O}(\Omega) \cap C(\overline{\Omega})$ by
$$
f(z) = {1 \over \sqrt{\pi v(E)}}\int_{E} {1 \over w - {1 \over z - z_0}}dv(w).$$
Then $f(z_0) = 0$ and $f^{\prime}(z_0) = -\sqrt{\pi^{-1}v(E)}$.  The Ahlfors-Beurling inequality states that $|f(z)| \leq 1$ for all $z \in \mathbb{C}_{\infty}$.  By the maximum modulus theorem, $|f(z)| < 1$ on $\Omega$.  Thus, $f$ is an extremal function for $c_B$, which implies $\sup_{z \in \Omega} |f(z)| = 1$.  By continuity there exists $z_2 \in \partial \Omega$ such that $|f(z_2)| = 1$.  We observe, after a careful inspection of the proof of the Ahlfors-Beurling inequality, as given in \cite[Lemma 5.3.6]{R95}, that $E$ must be a union of a closed disk with a  closed measure-zero set.  For completeness we resupply the proof since this observation is not stated in the literature, cf. \cite{AB50, G72, R95}.  Indeed, after a rotation and translation of $E$ we may assume that 
\begin{equation}\label{Ahlfors-Beurling proof f(z2)}
1 = |f(z_2)| = {1 \over \sqrt{\pi v(E)}}\int_{E} {1 \over w} dv(w).
\end{equation}
Let $D: =\{w\in\mathbb C: Re w^{-1} >(2a)^{-1}\}$ be a disk such that $v(D)=v(E)$. Then  $v(E\setminus D) = v(D\setminus E)$ and so
\begin{equation}\label{v2}
    \int_{E\setminus D}Re\frac{1}{w}dv(w)\leq  \int_{E\setminus D}\frac{1}{2a}dv(w) = \int_{D\setminus E}\frac{1}{2a}dv(w)\leq \int_{D\setminus E}Re\frac{1}{w}dv(w). 
\end{equation} 
This implies
\begin{equation*}
    \begin{split}
        \int_E \frac{1}{w}dv(w) =  \int_E Re \frac{1}{w}dv(w)=&\int_{E\cap D} Re\frac{1}{w}dv(w)+ \int_{E\setminus D} Re\frac{1}{w}dv(w)\\
        \leq & \int_{D\cap E} Re\frac{1}{w}dv(w)+ \int_{D\setminus E} Re\frac{1}{w}dv(w)\\
        =& \int_{D} Re\frac{1}{w}dv(w)\\
        =&\int_{-\frac{\pi}{2}}^{\frac{\pi}{2}}\int_0^{2a\cos\theta} \cos\theta dr d\theta = \pi a=\sqrt{\pi v(E)}.
    \end{split}
\end{equation*}
Comparing with \eqref{Ahlfors-Beurling proof f(z2)}, both inequalities in \eqref{v2} become equalities and thus $v(E \setminus D) = v(D \setminus E) = 0$. Since $E$ is compact, $\overline{D} \subset E$.  Consequently,  $\Omega = \mathbb{D}(z_1, r) \setminus P$ for some $z_1\in \mathbb C$ and $r = \sqrt{v(\Omega)\pi^{-1}}$, where  $P$ is a relatively closed set of measure 0.

To further show that $P$ satisfies \eqref{NB}, consider $$h(z) = \frac{r(z-z_0)}{r^2 - (\overline{z}_0 - \overline{z}_1)(z - z_1)}, \ \ z\in\Omega.$$ It is not hard to verify that $h\in \hbox{Hol}(\Omega, \mathbb D)$, $h(z_0)=0$ and $$h'(z_0) = \frac{r}{r^2-|z_0-z_1|^2}.$$ Thus $h$ is an extremal function for $c_B$ at $z_0$, by \eqref{analytic capcacity of general disk}. By \cite[Theorem 28]{H64}, the image of the extremal function  satisfies $h(\Omega) = \mathbb{D} \setminus Q$ where
$$
Q = h(P), \quad Q \cap \overline{\mathbb{D}(0, r)} \in \mathcal{N}_B, \quad \hbox{for all }0 \leq r < 1.
$$ 
Thus \eqref{NB} is proved.

If additionally \eqref{cbe} holds, then by Theorem \ref{main2}, $h$ is a biholomorphism from $\Omega$ to $\mathbb{D}(0, 1)\setminus Q$, where $Q$ is a relatively closed polar set. This implies that $P$ is also a relatively closed polar set.
Conversely if $\Omega = \mathbb{D}(z_1, r)\setminus P$ for  a relatively closed polar set $P$, then a direct computation shows $c_{\beta}^2(z_0) = \pi^{-1}{v(\mathbb{C}_{\infty} \setminus j_{z_0}(\Omega))}$.  The proof   is complete.

\end{proof}

The proof of Theorem \ref{jth} indicates the center of $\Omega$ may not necessarily be $z_0$, at which the equality \eqref{analytic capacity theorem main eqn 3} holds. Using the lemma below, we will show that if the stronger equality $c_B^2(z_0) = \pi v^{-1}(\Omega)$ holds for $v(\Omega)<\infty$, then in the conclusion of the preceding theorem the center of $\Omega$ must be $z_0$.

\begin{lemma} \label{Volume Lemma}
Let $\Omega \subset \mathbb{C}$ be a domain with $v(\Omega) < \infty$.  Then for all $z\in \Omega$,
\begin{equation} \label{Volume Lemma Inequality}
{\pi \over v(\Omega)} \leq {v(\mathbb{C}_{\infty} \setminus j_{z}(\Omega)) \over \pi},
\end{equation}
and equality holds at some $z_0\in \Omega$ if and only if $\Omega$ is a disk centered at $z_0$ less a relatively closed set of measure zero.
\end{lemma}

\begin{proof} If $B \subset \mathbb{C}$ is a set with $v(B) = 0$, then $v(j_{z_0}(B)) = 0$.  With this fact it is straightforward to verify that equality holds for a disk less a relatively closed set of measure 0.  Since $v(\Omega) = v(\Omega - \{z_0\})$ and $j_{z_0}(\Omega) = j(\Omega - \{z_0\})$, without loss of generality we may suppose $z_0 = 0$.
\medskip
Notice that for any $r>0$, $v(r\Omega) = r^2v(\Omega)$ and \begin{eqnarray*}
v(\mathbb{C}_{\infty} \setminus j(r\Omega))=v(\mathbb{C}_{\infty} \setminus r^{-1}j(\Omega)) =v(r^{-1}(\mathbb{C}_{\infty} \setminus j(\Omega))) = r^{-2}v(\mathbb{C}_{\infty} \setminus j(\Omega)).
\end{eqnarray*}
Here for any set $B\subset \mathbb C$,  $rB: =\{rz:z\in B\}$. We may further assume that $v(\Omega)=\pi$. So the inequality (\ref{Volume Lemma Inequality}) is equivalent to
\begin{equation*}
   v(j(\mathbb{C}_{\infty} \setminus \Omega)) =  v(\mathbb{C}_{\infty} \setminus j(\Omega))\ge \pi.
\end{equation*}

  Set 
$$
S_1 =   {\mathbb D} \setminus \Omega, \quad S_2 =  \Omega\setminus {\mathbb D}.
$$
So $v(S_1)=v(S_2)$.
Since $\mathbb{C}_{\infty} \setminus \Omega = S_1\sqcup ((\mathbb{C}_{\infty} \setminus {\mathbb D})\setminus S_2)$ and   $j(\mathbb C_\infty\setminus {\mathbb D}) =\overline {\mathbb D}$, 
$$ j(\mathbb{C}_{\infty} \setminus \Omega) =   j(S_1) \sqcup (\overline {\mathbb D}\setminus j(S_2)), $$
where $\sqcup$ denotes the disjoint union. Noticing  $j(S_2)\subset \overline {\mathbb D}$, we further have $$v(j(\mathbb{C}_{\infty} \setminus \Omega)) -\pi  =v(j(S_1)) +v(\overline {\mathbb D})-v(j(S_2))-\pi = v(j(S_1))  -v(j(S_2)).$$  Applying change of variables formula, one gets
\begin{equation}\label{vv}
    \begin{split}
    v(j(\mathbb{C}_{\infty} \setminus \Omega)) -\pi  =\int_{S_1} \frac{1}{|z|^4}  dv(z) - \int_{S_2} \frac{1}{|z|^4}  dv(z)\ge  \int_{S_1} 1  dv(z) - \int_{S_2}1  dv(z) =0.
    \end{split}
\end{equation}
Here we have used the fact that $|z|< 1$ on $S_1$ and $|z|\ge 1$ on $S_2$. This completes the proof of the inequality part. 

If equality holds in (\ref{Volume Lemma Inequality}), then the inequality in (\ref{vv}) becomes equality and 
$$\int_{S_1} \frac{1}{|z|^4}  dv(z) =\int_{S_1} 1  dv(z), \ \  \int_{S_2} \frac{1}{|z|^4}  dv(z) = \int_{S_2}1  dv(z).$$ 
Since $|z| < 1$ on $S_1$, the first equation implies  $v(S_1)=0$. Thus, $v(S_2)=0 $. By definitions of $S_1$ and $S_2$, we know that $\Omega$ is the unit disk centered at $0$ less a relatively closed set of measure zero.

\end{proof}

\begin{proof}[{\bf Proof of Theorem \ref{rigidity of analytic and logarithmic capacity with reciprocal of volume}} (Rigidity theorem of $c_{B}$ and $c_{\beta}$)] If $v(\Omega) = \infty$, then the inequality is trivial.  By \cite[p. 107]{AB50}, $c_B(z_0) = 0$ if and only if $c_B \equiv 0$, and thus if and only if $\Omega = \mathbb{C}_{\infty} \setminus P$ where $P \in \mathcal{N}_B$ by definition. 

We now assume $v(\Omega) < \infty$. Equation \eqref{analytic capacity theorem main eqn} follows from \eqref{analytic capacity theorem main eqn 2} and  \eqref{Volume Lemma Inequality}. If $\Omega = \mathbb{D}(z_0, \sqrt{\pi^{-1}v(\Omega)})\setminus Q$, where
$$
Q \cap \overline{\mathbb{D}(z_0, s)} \in \mathcal{N}_B, \quad \hbox{for all }s < \sqrt{\pi^{-1}v(\Omega)},
$$
then by the classical Schwarz lemma, $c_B^2(z_0) = \pi v(\Omega)^{-1}$. 

Conversely, if equality holds at $z_0\in \Omega$, then  
$$ c_B^2(z_0) = {v(\mathbb{C}_{\infty} \setminus j_{z_0}(\Omega)) \over \pi}  = {\pi \over v(\Omega)}. $$
By the equality part of Lemma  \ref{Volume Lemma}, $\Omega = \mathbb{D}(z_0, \sqrt{\pi^{-1}v(\Omega)}) \setminus P$, where $P$ is a relatively closed set of measure 0. By the equality part of Theorem \ref{jth}, we further conclude that $
P \cap \overline{\mathbb{D}(z_0, s)} \in \mathcal{N}_B$ for all $s < \sqrt{\pi^{-1}v(\Omega)}.$ 
The rest of the theorem follows from the second part of Theorem \ref{jth}.

\end{proof}

\begin{remark} \label{EX}
If $c_{\beta}^2(z_0) > c_B^2(z_0) = \pi v^{-1}(\Omega)$, then $P$ in the preceding theorem may not be polar. 
Indeed, let $Q$ be the compact four-corner Cantor set defined in \cite{GY01}.  As shown therein, $Q \in \mathcal{N}_B,$ but is not polar.  Let $\Omega = \mathbb{D}(z_0, r) \setminus Q$ where $z_0$ and $r$ are chosen such that $z_0 \not\in Q \subset \mathbb{D}(z_0, r)$. Since $Q \in \mathcal{N}_{B}$, all bounded holomorphic functions on $\Omega$ extend across $Q$.  Thus,
$$
c_{B; \Omega}(z_0) = c_{B; \mathbb{D}(z_0, r)}(z_0) = \sqrt{\pi \over v(\mathbb{D}(z_0, r))} = \sqrt{\pi \over v(\Omega)},$$
where the last equality used the fact that sets of class $\mathcal{N}_B$ have two-dimensional Lebesgue measure 0.
 
\end{remark}

\section{Applications of the Rigidity Theorem of $c_{B}$ and $c_{\beta}$ } \label{Section five}

Let $\Omega$ be a domain in $\mathbb C^n, n\ge 1$. The Bergman space of a domain $\Omega$ is the Hilbert space
$$
A^2(\Omega) = L^2(\Omega) \cap \mathcal O(\Omega)
$$
with $L^2(\Omega)$-norm denoted by $\|\cdot\|_\Omega$.   The Bergman kernel is defined by
 $$
K(z) = \sup\{|f(z)|^2: f \in A^2(\Omega), \|f\|_\Omega \leq 1\}, \ \  z\in \Omega.
 $$
By considering constant functions in the defining set of the kernel, we get
 $$
K(z) \geq \frac{1}{v(\Omega)}, \  z\in \Omega,
 $$
which  is sharp when $\Omega = \mathbb{D}(z_0, r) \setminus P$, where $P$ is a relatively closed polar set and $z = z_0$.

\medskip

   As an application  of  Theorem \ref{rigidity of analytic and logarithmic capacity with reciprocal of volume}, we show that this sharp example is
   in fact the only possible domain where the equality can be achieved. Corollary \ref{theorem 1} below, which is the main result of \cite{DT}, was originally proved by the first and second named authors using the equality part of the Suita Conjecture. Here we provide a new    
    proof 
   based on our Theorem \ref{rigidity of analytic and logarithmic capacity with reciprocal of volume} and  Suita's Theorem.

\begin{corollary}[Originally proved in \cite {DT}] \label{theorem 1}
Let $\Omega \subset \mathbb{C}$ be a domain.  Then there exists a $z_0 \in \Omega$ such that
\begin{equation}\label{Minimal point of bergman kernel}
K(z_0) = {1 \over v(\Omega)},
\end{equation}
if and only if either of the following holds true:

\begin{enumerate}

\item [$1$.] $v(\Omega) < \infty$ and $\Omega = \mathbb{D}(z_0, r) \setminus P$, where $P$ is a  relatively closed polar set and with $r^2 =\pi^{-1} v(\Omega)$.

\item [$2$.] $v(\Omega) = \infty$ and $\Omega = \mathbb{C} \setminus P$,  where $P$ is a possibly empty, closed polar set.

\end{enumerate}
\end{corollary}

\begin{proof}
First assume $v(\Omega) < \infty$.  By Suita's Theorem (Theorem \ref{Suita Theorem}) and Theorem \ref{rigidity of analytic and logarithmic capacity with reciprocal of volume},
\begin{equation*}
\pi K(z_0) \geq c_B^2(z_0)  \geq {\pi \over v(\Omega)} = \pi K(z_0).
\end{equation*}
By the equality part of Theorem \ref{rigidity of analytic and logarithmic capacity with reciprocal of volume},
$$
\Omega = \mathbb{D}(z_0, r) \setminus P, \quad P \cap \mathbb{D}(z_0, s) \in \mathcal{N}_B, \quad s < r = \sqrt{\pi^{-1}v(\Omega)}.
$$ The equality part of Suita's Theorem (Theorem \ref{Suita Theorem}) implies additionally that %a relatively 
$P$ is polar. 
The case when $v(\Omega) = \infty$ is already known (see \cite[Theorem 4]{BlZw}).  

\end{proof} 

\begin{remark} The case $v(\Omega) < \infty$ in the preceding proof used Suita's Theorem and not the Suita Conjecture.  The proof of Suita's Theorem, which was proved using Riemann surface theory, is much simpler than the proof of the Suita Conjecture.  Thus, the proof given here 
is simpler than the original proof in \cite{DT}.

\end{remark}

\begin{remark}
Boas \cite{Bo22} has kindly pointed out to us that Corollary \ref{theorem 1} fails for domains in $\mathbb C^n, n\ge 2$, as indicated by the following examples.  Let $\Omega$ be a domain in $\mathbb C^n, n\ge 2$, satisfying \eqref{Minimal point of bergman kernel} (for instance, a ball centered at $z_0$), and consider images $F(\Omega)$ under maps of the form $F:\mathbb{C}^n \to \mathbb{C}^n$ defined by
\begin{equation}\label{shear}
F(z) = (z^{\prime}, z_n + f_0(z^{\prime})), \quad z^{\prime} = (z_1, \ldots, z_{n-1}),
\end{equation}
where $f_0:\mathbb{C}^{n-1} \to \mathbb{C}$ is any holomorphic function. Such maps are called shears and were considered extensively by Rosay and Rudin \cite{RR88}.  $F$ is biholomorphic,   and the determinant  $\det J_{\mathbb{C}}F$ of  its complex Jacobian is constantly 1.  Hence $F$ is volume preserving with  $v(F(\Omega)) = v(\Omega) $. On the other hand,  by the transformation rule of the Bergman kernel, $K_{F(\Omega)}(F(z_0)) =K_{\Omega}(z_0)$, so
\eqref{Minimal point of bergman kernel} holds for $F(\Omega)$ at $F(z_0)$.  Due to the arbitrariness of  $f_0$, a large degree of freedom is afforded to the geometries of such $F(\Omega)$, in stark contrast to the situation in $\mathbb{C}$.  Moreover, the domains $\Omega$ that satisfy \eqref{Minimal point of bergman kernel} include the bounded complete Reinhardt domains,  bounded complete circular domains,  bounded quasi-circular domains containing the origin, and bounded quasi-Reinhardt domains containing the origin \cite{LR19}, which form a strictly increasing sequence under the set containment relation in $\mathbb{C}^n, n \geq 2$.  By selecting $\Omega$ from one of these classes, we can produce additional domains $F(\Omega)$ satisfying \eqref{Minimal point of bergman kernel} that are not biholomorphically equivalent to the ball. Lastly, for a bounded quasi-Reinhardt $\Omega$ containing the origin, by choosing a non-polynomial mapping $F$  in $\eqref{shear}$  such that $F(\Omega)$ is bounded and F(0) = 0, we may get a minimal domain centered at 0  
that is not quasi-Reinhardt (see \cite{DR16, LR19}).

\end{remark}

   For domains in $\mathbb{C}$, we
  also have the following more precise estimate on the on-diagonal Bergman kernel. Recall that given $z\in \Omega$, $j_{z} = {1 \over \cdot - z}$.

 \begin{corollary} \label{Corollary}
 Let  $\Omega \subset \mathbb{C}$ be a domain with $v(\Omega) < \infty$.  Then for all $z\in \Omega$,
 $$
 K(z) \geq {v(\mathbb{C}_{\infty} \setminus j_{z}(\Omega)) \over \pi^2},
 $$
and   equality holds at some $z_0\in \Omega$ if and only if $\Omega$ is a disk less a relatively closed polar set.
 \end{corollary}

\begin{proof} By   
Suita's Theorem (Theorem \ref{Suita Theorem}), the Ahlfors-Beurling Inequality \eqref{analytic capacity theorem main eqn 2}, and Theorem \ref{jth},
$$
\pi K(z) \geq c_B^2(z)  \geq {v(\mathbb{C}_{\infty} \setminus j_{z}(\Omega)) \over \pi}
$$
for $z\in \Omega$. The equality part is a consequence of Theorems \ref{Suita Theorem} and \ref{jth}.

\end{proof}

For $j = 1, 2,\ldots$, let
$$
K^{(j)}(z) = \sup\{|f^{(j)}(z)|^2: f \in A^2(\Omega), \|f\|_{\Omega} \leq 1, f^{(k)}(z) = 0, k = 0,\ldots, j-1\}
$$
 denote the Bergman kernels for higher order derivatives, and set $K^{(0)} = K$.  B\l{}ocki and Zwonek \cite{BlZw} established that for $z \in \Omega \subset \mathbb{C}$,
\begin{equation} \label{Higher order Bergman kernel inequality}
K^{(j)}(z) \geq {j!(j+1)! \over \pi}(c_{\beta}(z))^{2j + 2},
\end{equation}
which is sharp for a disk less a relatively closed polar set $\mathbb{D}(z_0, r) \setminus P$ with $z = z_0$.  

\begin{proof}[{\bf Proof of Corollary \ref{Higher Bergman Kernel Corollary}} (Rigidity theorem of higher-order Bergman kernels)]  The case $v(\Omega) = \infty$ is already known, cf. \cite[Theorem 4]{BlZw}. For $v(\Omega) < \infty$, this follows from \eqref{Higher order Bergman kernel inequality} and Theorem \ref{rigidity of analytic and logarithmic capacity with reciprocal of volume}.  

\end{proof}

 One important property of the analytic capacity $c_B $ is that it is distance decreasing with respect to  any given holomorphic map $f:X  \to Y$,
where $X, Y$ are hyperbolic Riemann surfaces, cf. \cite[Chapter 2]{JP}:
\begin{equation} \label{c_B decreasing}
f^*\left(c_{B; Y}(z) |dz| \right) \leq c_{B; X}(z) |dz|,
\end{equation}
where $f^*\left(c_{B; Y}(z) |dz| \right)$ denotes the pull-back to $X$ via $f$ of the analytic capacity on $Y$.

   \begin{remark} 

For a hyperbolic Riemann surface $X$ with a local coordinate $z$, if there exists $z_0 \in X$ and a non-constant holomorphic function $f:X \to \mathbb D$   such that 
\begin{equation} \label{isom}
  K(z_0) |dz|^2 = \frac{|df (z_0)|^2  }{ \pi (1-|f(z_0)|^2)^2},
\end{equation}
 then by \eqref{Inequality of surface quantities} and \eqref{c_B decreasing}, we have
$$
\sqrt{\pi K(z_0)}|dz| \geq   c_{B; X}(z_0) |dz|  \geq  f^*\left(c_{B; \mathbb D}(z_0) |dz| \right)   = {|df(z_0)| \over 1 - |f(z_0)|^2 }.
$$
Without loss of generality, assume that $df (z_0)>0$.
Therefore, \eqref {isom} forces the equality condition in Theorem \ref{Suita Theorem} to hold true,  and there exists a  biholomorphism $h: X \to \mathbb D \setminus P$ such that $h(z_0) = 0$ and $\frac{\partial h}{\partial z} (z_0)>0$, where $P$ is a relatively closed polar subset.  By the transformation rule of the Bergman kernel,  
$\sqrt{\pi K(z_0)}  = \frac{\partial h}{\partial z} (z_0).$ Take 
$
\varphi (z):= \frac{z- f(z_0)}{ 1- \overline {f(z_0)} z} \in \text{Aut} (\mathbb D).
$
Then $ \varphi \circ f \circ h^{-1}$ is bounded, so it extends to a holomorphic function $F:  \mathbb D  \to  \mathbb D  $ 
such that $F (0)=  0$. Moreover, $F^{\prime}(0) 
= \sqrt{\pi K(z_0)}^{-1} \cdot  \sqrt{\pi K(z_0)}   (1-|f(z_0)|^2)  \cdot  (1-|f(z_0)|^2)^{-1} =1$.
 The Schwarz lemma implies that $F(z)  \equiv z$ on $  \mathbb D $, so $f $ is in fact  biholomorphic. Consequently, the identity \eqref{isom} extends to all of $X$, and $f$ is a holomorphic isometry with respect to the Bergman kernel.

 \end{remark}

In the remaining part of the section, we focus on 
a bounded domain $\Omega$ in $\mathbb C$ with Lipschitz boundary. The Hardy space $H^2(\partial \Omega)$ is  defined to be the set of holomorphic functions on $\Omega$ which admit a non-tangential boundary limit function and a maximal function on the boundary which belong to $L^2(\partial \Omega)$. See Appendix \ref{AppendixC} or \cite{L99} for more details.  Let $\|\cdot\|_{\partial\Omega}$ denote the $L^2(\partial \Omega)$ norm and $S(\cdot, \cdot)$ denote the Szeg\"o kernel.  When restricted to the diagonal, it satisfies
 $$
S(z, z) := S(z) = \sup\{|f(z)|^2: f \in H^2(\partial\Omega), \|f\|_{\partial\Omega} \leq 1\}.
 $$
By considering constant functions in the defining set  of the  kernel,  for any $ z \in \Omega$, we get  
$$
 S(z) \geq {1 \over \sigma(\partial \Omega)},
$$
where $\sigma(\partial \Omega)$ denotes the arclength of $\partial\Omega$.  The lower bound is sharp for $\Omega = \mathbb{D}(z_0, r)$ and $z = z_0$. As applications of the Rigidity theorem of $c_B$ and $c_{\beta}$, Theorem \ref{rigidity of analytic and logarithmic capacity with reciprocal of volume}, we will show that these are the only possible domains where the equality can be achieved.

\medskip

Let $\{\Omega_j\}_{j=1}^\infty$ be a family of exhausting subdomains of $\Omega$, and  $S_j(\cdot, \cdot)$ %and $S(\cdot, \cdot)$ 
be the corresponding Szeg\"o kernels. Boas showed in \cite[Theorem 2.1]{Bo87} that if in addition $\Omega$ has $C^{\infty}$-smooth boundary and is exhausted by sublevel sets $\Omega_j$ of its defining function, %(and $\Omega_j$ are sublevel sets of the defining function), 
then  for $a, z \in  \Omega_{j}$,
 $$
\lim_{j \to \infty} S_{j}(z, a) = S(z, a).
 $$
For a bounded domain $\Omega \subset \mathbb{C}$ with Lipschitz boundary, there are subdomains $\{\Omega_j\}_{j=1}^{\infty}$ with $C^{\infty}$ boundary that approximate $\Omega$ well uniformly and non-tangentially in the sense of Ne\v{c}as (see \cite[p. 539]{L99} or Appendix \ref{AppendixC} for more details). In Proposition \ref{Our Szego Lemma}, we extend  Boas' stability result of the Szeg\"o kernel to Lipschitz domains in $\mathbb{C}$ with respect to the Ne\v{c}as approximation.

It is known for a finitely connected domain with $C^{\infty}$ boundary that the analytic capacity and (on-diagonal) Szeg\"{o} kernel satisfy the relation $c_B(z) = 2\pi S(z)$.   For $\{\Omega_j\}_{j=1}^{\infty}$ given above, as shown by Ahlfors and Beurling \cite[Theorem 1]{AB50}, the analytic capacities $c_{j; B}$ and $c_B$ of these domains, respectively, similarly satisfy
\begin{equation*}
\lim_{j \to \infty} c_{j; B}(z) = c_B(z), \quad z \in \Omega.
\end{equation*}
Consequently, we can conclude

\begin{proposition}\label{Relation with analytic capacity and szego kernel}
If $\Omega \subset \mathbb{C}$ is bounded with Lipschitz boundary, then 
\begin{equation}\label{equation relating analytic capacity and Szego}
c_B(z) = 2\pi S(z), \quad z \in \Omega.    
\end{equation}

\end{proposition}

The Rigidity theorem of $c_B$ and $c_{\beta}$, Theorem \ref{rigidity of analytic and logarithmic capacity with reciprocal of volume}, together with Proposition \ref{Relation with analytic capacity and szego kernel},  enables us to prove the rigidity phenomenon of the Szeg\"o kernel.

\begin{proof}[{\bf Proof of Theorem \ref{NEW}} (Rigidity theorem of the Szeg\"o kernel)]

The equality conditions in Suita's Theorem (Theorem \ref{Suita Theorem}) and in $c_\beta \geq c_B$ (Theorem \ref{main2}) imply that the equality condition in Case 1 holds if and only if $\Omega$ is biholomorphic to a disk less a relatively closed polar set.  Since $\Omega$ is Lipschitz, it is a regular domain for the Dirichlet problem.  Thus, the polar set is empty, cf. \cite{R95}.  

For Case 2, it suffices to prove the `only if' direction. By the Rigidity theorem of $c_B$ and $c_{\beta}$ (Theorem \ref{rigidity of analytic and logarithmic capacity with reciprocal of volume}), $2\pi S(z_0) = \sqrt{\pi v^{-1}(\Omega)}$ if and only if
$$
\Omega = \mathbb{D}(z_0, r) \setminus Q, \quad Q \cap \overline{\mathbb{D}(z_0, s)} \in \mathcal{N}_B, \quad \hbox{for all }s < r = \sqrt{\pi^{-1}v(\Omega)}.
$$
Since $\Omega$ has Lipschitz boundary, its boundary has no singleton connected components.  Thus $Q = \emptyset.$
The equality case $2\pi S(z_0) = 2\pi \sigma^{-1} (\partial \Omega) $ follows now as well. 

By \eqref{equation relating analytic capacity and Szego}, $2\pi S(z_0) = \delta^{-1} (z_0) $ is equivalent to $c_B(z_0) = \delta^{-1}(z_0)$.  Let $f_0$ be an extremal function for $c_B(z_0)$.  By the Schwarz lemma, $f_0|_{\mathbb{D}(z_0, \delta(z_0))} : \mathbb{D}(z_0, \delta(z_0)) \to \mathbb{D}$ is extremal  for $\mathbb{D}(z_0, \delta(z_0))$, and thus equals $e^{i\theta}\delta^{-1}(z_0)(z - z_0)$ for some $\theta \in [0, 2\pi)$.  Since $|f_0| < 1$, $\Omega = \mathbb{D}(z_0, \delta(z_0)).$ 

\end{proof}

Theorem \ref{jth} combined with Proposition \ref{Relation with analytic capacity and szego kernel} also gives another rigidity result concerning the Szeg\"o kernel below.

\begin{corollary}
Suppose $\Omega \subset \mathbb{C}$ is a bounded domain with Lipschitz boundary. Then for $z\in\Omega$,
 $$
 S^2(z) \geq   {v(\mathbb{C}_{\infty} \setminus j_{z}(\Omega)) \over 4\pi^3}.
 $$
 Moreover, equality holds at some $z_0\in \Omega$ if and only if  $\Omega$  is a disk.
\end{corollary}
\medskip

\begin{remark}
In comparison with \eqref{Lip Domain Inequality Chain} and Theorem \ref{NEW}, concerning the inequality between the on-diagonal Bergman and Szeg\"o kernels and the equality conditions, we would like to point out that when the domain is bounded and simply-connected  with $C^{\infty}$-boundary, 
$$
K(z, a) = 4\pi S^2(z, a), \quad z, a \in \Omega, 
$$
cf. \cite[Theorem 25.1]{Be16}.  
\end{remark}

\begin{remark} \label{decS}

Let $f:\Omega_1 \to \Omega_2$ be a holomorphic map, where $\Omega_1, \Omega_2 \subset \mathbb{C}$ are bounded domains with Lipschitz boundaries. Then, by Proposition \ref{Relation with analytic capacity and szego kernel} and \eqref{c_B decreasing},
the Szeg\"o kernel is decreasing with respect to $f$, namely
$$
|f^{\prime}(z)|S_{\Omega_2}(f(z)) \leq S_{\Omega_1}(z),
$$
and equality holds if $f$ is a biholomorphism.

\end{remark}

 \section{Rigidity of Sublevel Sets of Green's Function} \label{Section 6}

For a fixed $z_0 \in \Omega$, let $G_{z_0}(\cdot) = G(\cdot, z_0)$ and
 $$
 \Omega_t = \{z \in \Omega: G_{z_0}(z) < t\}, \quad t \in (-\infty, 0]
 $$
denote the sublevel sets of the Green's function. The following monotonic property was proved by B\l{}ocki and Zwonek.

\begin{theorem}{\normalfont\cite{BZ15}}\label{BZ1}
Let $\Omega$ be a bounded domain in $\mathbb C$. Then for any $z_0\in \Omega$,
$$f(t): = \frac{\pi e^{2{t}}}{v(\Omega_t)}$$
 is non-increasing in $t\in (-\infty, 0)$.
\end{theorem}

\begin{remark}\label{mon}
Note that  $$\lim_{t\rightarrow 0^-} \frac{e^{2t}}{ v{(\Omega_t)}} = \frac{1}{v(\Omega)}.$$
As $t\rightarrow -\infty$,
$\Omega_t$ is approximable by the set $\{\log \left(c_{\beta}(z_0)|z-z_0|\right)<t \} $ (see \cite{BL, BZ15}). This implies   $$\lim_{t\rightarrow -\infty} \frac{e^{2t}}{ v{(\Omega_t)}} =\lim_{t\rightarrow -\infty} \frac{e^{2t}}{ v{(\{|z-z_0|<  {e^t c_{\beta}^{-1}(z_0) } \})}} =\frac{c_{\beta}^2(z_0)}{\pi}.$$
This gives the inequalities (\ref{inequalities with log capacity and sublevel sets of green's function}) between the logarithmic capacity and the sublevel sets of the Green's function.
\end{remark}

The next theorem discusses the second inequality  in \eqref{inequalities with log capacity and sublevel sets of green's function}, which states that either $f$ in Theoerem \ref{BZ1} is strictly decreasing, or the domain has to be rigid.

\begin{theorem}\label{ll}
Let $\Omega$ be a bounded domain in $\mathbb C$. If there exist $z_0\in \Omega$ and $t_1 < t_2$ in $(-\infty, 0)$ such that  $$ \frac{\pi e^{2{t_1}} }{v{(\Omega_{t_1})}}   = \frac{\pi e^{2{t_2}} }{v(\Omega_{ t_2})},$$ then  $G_{z_0}(z) = \log |z-z_0|/R$ for some  constant $R>0$. Consequently,
$\Omega$ is a disk centered at $z_0$ with radius $R$ possibly less a relatively closed polar  subset.
\end{theorem}

\begin{proof}

By scaling and translating if necessary,  we may assume that $z_0=0$  and the diameter of $\Omega$ is $2$.
By the inequality part of Theorem \ref{BZ1}, $f(t_1)=f(t_2)$ implies that there exists a constant $C$ such that
\begin{equation*}\label{*}
v{(\Omega_t)}= C e^{2t},
\end{equation*}
for $t\in(t_1, t_2)$. Thus
\begin{equation*}\label{con1}
\frac{d}{dt} v{(\Omega_t)} = 2v{(\Omega_{t})}.
\end{equation*}

 Following \cite{BZ15, BlZw}, we see from the Cauchy-Schwarz inequality that for
  almost
  every $t\in (-\infty, 0)$,
\begin{equation}\label{CS}
     \sigma^2(\partial \Omega_t)   \le  \int_{\partial \Omega_t } \frac{1}{|\nabla {G_0}|}d\sigma\int_{\partial \Omega_t } {|\nabla {G_0}|}  d\sigma.
\end{equation}
Note that due to the harmonicity of $G_0$ away from $0$, \begin{equation}\label{int}
    \int_{\partial \Omega_t } {|\nabla {G_0}|}  d\sigma = \int_{\partial \Omega_t } \frac{\partial  {G_0}}{\partial \nu}  d\sigma = 2\pi.
\end{equation}
On the other hand, by the co-area formula  $v(\Omega_t) =\int_{-\infty}^t\int_{\partial \Omega_s} \frac{d\sigma}{|\nabla {G_0}|} ds  $, which further  leads to
\begin{equation}\label{co}
    \int_{\partial \Omega_t } \frac{1}{|\nabla {G_0}|}d\sigma = \frac{d}{dt}v(\Omega_t).
\end{equation}
for almost every $t\in (-\infty, 0)$. Pick up a point $t_0\in (t_1, t_2)$ such that (\ref{CS}-\ref{co}) hold. Then,
$$  \sigma^2(\partial \Omega_{t_0})  \le
4\pi v(\Omega_{t_0}).  $$

According to the classical isoperimetric inequality (see  \cite{De} and the references therein),  $ \Omega_{t_0} $ is equivalent (two sets $E_1$ and $E_2$ are equivalent if and only if $v(E_1\cup E_2 \setminus E_1\cap E_2)=0$ ) to a disk  centered at a point $a\in \Omega$ with radius $re^{t_0}$ for some $r> 0$, and in particular, the involved Cauchy-Schwarz inequality (\ref{CS}) attains the equality. This means  $\frac{1}{|\nabla {G_0}| }$ is necessarily a constant multiple of $|\nabla {G_0}| $, or equivalently, $|\nabla {G_0}|  $ is  a constant on $\partial \Omega_{t_0}$. Combining this with (\ref{int}), one obtains $\frac{\partial  }{\partial \nu} G_0= |\nabla {G_0}| = r^{-1}e^{-t_0}$ on $\partial \Omega_{t_0}$.

Now   $G_0$ is harmonic on $\Omega \setminus \Omega_{t_0}$ and  satisfies the following Cauchy data  $$G_0 = t_0,  \ \ \frac{\partial  G_0}{\partial \nu}=\frac{1}{re^{t_0}}$$ on some smooth piece in $\partial \Omega_{t_0}$ contained in $\partial \mathbb{D}(a, re^{t_0})$. As a consequence of the unique continuation property of harmonic functions for local Cauchy data (see \cite{Ta}), we get $G_0(z) = \log \frac{|z-a|}{r}$ on $\Omega \setminus \Omega_{t_0}$, and further on $\Omega$ by the uniqueness. Since $G_0$ has a pole $0$ and  the diameter of $\Omega$ is $2$, we see that $a=0$ and  $r=1$, with $G_0(z) = \log {|z|}$ on $\Omega$. Then we complete the  proof using Lemma \ref{main}.

 \end{proof}

In particular, the  corollary below follows  from Theorem \ref{ll} directly. Here we  provide an alternative proof adopting the property of the Bergman kernel, instead of using the unique continuation property of harmonic functions as in Theorem \ref{ll}.

\begin{corollary} \label{weak}
Let $\Omega$ be a bounded domain in $\mathbb C$. Then,  there exist $z_0\in \Omega$ and  $t_0\in (-\infty, 0]$ such that
$$
c_{\beta}^2(z_0)  = \frac{\pi e^{2{t_0}} }{v{(\Omega_{ t_0})}},
$$
if and only if $\Omega$ is a disk centered at $z_0$ possibly less a relatively closed polar subset. In particular,
$$
c_{\beta}^2(z_0)  = \frac{\pi}{v(\Omega)}$$ for
some $z_0\in \Omega$
if and only if $\Omega$ is a disk centered at $z_0$ possibly less a relatively closed polar subset.
\end{corollary}

\begin{proof}  It suffices to prove the necessity. By scaling and translating, we may assume $z_0=0$ and $c_{\beta}(0) =1$. By the isoperimetric inequality as in the first part of the proof of Theorem \ref{ll},  $ \Omega_t $ is equivalent to a disk of radius $e^{t}$ for almost every $t\in (-\infty, t_0)$.
Let  $t^\sharp<t_0$ be a negatively large constant  such that   when  $t<t^\sharp$, $\partial \Omega_t$ is smooth.
Hence $\Omega_t$ is precisely a disk of radius $e^t$ for each such arbitrarily fixed $t$. Denote by $a_t$ the center of $\Omega_t$.

We claim that  $a_t=0$.
To see this, let $K_t$ and $c_t$ stand for the corresponding Bergman kernel and logarithmic capacity of $\Omega_t$, respectively. Then for $z\in \Omega_t$,
$$K_t(z) = \frac{e^{2t}}{\pi(e^{2t}-|z-a_t|^2)^2}.$$
Since the Green's function on $\Omega_t$ with a pole 0 is $G_0(z)-t$,  by definition,
$$c_t(0)=\exp \lim_{z\to {0}} \left\{(G_0(z)-t) - \log |z|\right\} = e^{-t}c(0) = e^{-t}.$$
 Making use of the fact that  $\pi K_t =c_t^2$ on the disk $\Omega_t$,
we further have at $z=0$ that
$$ \frac{e^{2t}}{(e^{2t}-|a_t|^2)^2}= \pi K_t(0)=c_t^2(0)  = e^{-2t}.$$
It immediately tells us that $a_t=0$, so 
\begin{equation}\label{ot}
    \Omega_t =\mathbb D(0, e^t), \ \ \text{for}\ \ t<t^\sharp.
\end{equation}

Lastly  let $\rho(z) := G_0 (z) -\log |z|$ on $\Omega$. Then   $\Omega_t = \{z\in \mathbb C: |z|<\frac{e^t}{e^{\rho(z)}}\},\ t<0$.  Comparing this with (\ref{ot}) for $t<t^\sharp$, we have $\rho|_{\partial \Omega_t} =0 $. Since $\rho$ is harmonic  on $\Omega$,  $\rho\equiv 0 $ on $\Omega_t$.  By the uniqueness again, $\rho\equiv 0$  on $\Omega$, where $G_0 (z)=\log |z|$.
 The proof  is complete in view of  Lemma \ref{main}.

\end{proof}

\begin{proof}[{\bf Proof of Theorem \ref{lc}.} ] The `only if' direction is straightforward. For the `if' direction,  Case 2 follows from Theorem \ref{ll}. Case 1 follows from Corollary \ref{weak}, or alternatively, from Theorem  \ref{ll}. Indeed, if Case 1 holds, then by  \eqref{inequalities with log capacity and sublevel sets of green's function} we have  for all  $t<t_1$,
$$\frac{\pi e^{2{t}} }{v(\Omega_{t})}   = \frac{\pi e^{2{t_1}} }{v(\Omega_{t_1})}. $$
Case 1 is thus reduced to Theorem \ref{ll} and so $\Omega$ is a disk centered at $z_0$ possibly less a relatively closed polar subset. Case 3 can be proved similarly from Theorem \ref{ll}.

\end{proof}

\section{Rigidity Properties of the Distance Function in $\mathbb{C}^n$} \label{Section 7}

 Let $\Omega$ be a  bounded domain  in $\mathbb C^n$ with $C^2$-boundary.  It is known that  there exists a constant $C$ depending only on $n$ such that $$K(z)\le  {C}{ \delta^{-n-1}(z)}$$ for all $z \in   \Omega $;
if in addition  $\Omega$ is pseudoconvex, by \cite{OT, Fu}
there exists a constant $ C_\Omega$ depending only on $\Omega$ such that
$$K(z)\ge  \frac{ C_\Omega}{ \delta^{2}(z)}.$$
In particular,  the  Bergman kernel of $\mathbb B^n(z_0, r)$, the ball in $\mathbb C^n$ centered at $z_0$ with radius $r$,  is given by
\begin{equation}\label{bn}
K_{\mathbb B^n(z_0, r)}(z) = \frac{n!r^{2 }}{\pi^n (r^2-|z-z_0|^2)^{n+1}}.
\end{equation}

Denote by $PSH^{-}(\Omega)$  the space of negative plurisubharmonic functions on $\Omega$. Let $G_z(\cdot)$ be the pluricomplex Green's function of $\Omega$ with pole $z\in \Omega$  given by
 $$
G_z(w) = \sup\{u(w): u \in PSH^{-}(\Omega), \limsup_{\zeta \to z} u(\zeta) - \log|\zeta - z| < \infty\}.
$$ The Azukawa indicatrix  for $z\in \Omega \subset \mathbb C^n$ is defined as
\begin{equation}\label{125}
    I_{\Omega}^A(z): =\{X\in \mathbb C^n: \limsup_{\zeta\rightarrow 0} \left(G_z(z+\zeta X)-\log|\zeta|\right)<0\}.
\end{equation}
 Straightforward calculations show $I_{\mathbb{B}^n(z, r)}^A(z) = \mathbb{B}^n(0, r),$ and when $n=1$,  $I_{\Omega}^A(z) = \mathbb D (0, c_{\beta; \Omega}^{-1}(z) )$.

 \begin{proof} [{\bf Proof of Theorem \ref{delta-c}.} ]

Firstly,  we deal with the relation between  the Bergman kernel and the distance function.
Let $z\in \Omega$ be fixed and consider $\mathbb B^n(z, \delta(z))$. Then $\mathbb B^n(z, \delta(z))\subset \Omega$. By   (\ref{bn}) and the monotonic decreasing property of the Bergman kernels with respect to domains, 
\begin{equation*}
    K(z)\le K_{\mathbb B^n(z, \delta(z))}(z)\le \frac{n!}{\pi^n\delta^{2n} (z)},
\end{equation*}
which yields  \eqref{B}.

Then we prove the rigidity part in \eqref{B} and without loss of generality assume equality is attained at $z_0=0$, namely $K(0) = {n! }{\pi^{-n} \delta^{-2n} (0)} $. Let $f $ be an extremal holomorphic function on $\Omega$ such that $\|f\|_{L^2(\Omega)} =1$ and $K(0) = |f(0)|^2$. Then
the restriction of $f $ to $\Omega_1:=\mathbb B^n(0, \delta(0))$ is also holomorphic with $ \|f\|_{L^2(\Omega_1)}\le \|f\|_{L^2(\Omega)} =1 $.
Therefore, the Bergman kernel on the ball $\Omega_1$ satisfies

$$ \frac{n! }{\pi^n\delta^{2n }(0)} = K_{\Omega_1}(0) \ge \frac{|f(0)|^2}{\|f\|^2_{L^2(\Omega_1)}}\ge  |f(0)|^2 =K(0) = \frac{n! }{\pi^n \delta^{2n}(0)},
$$
which forces both inequalities above to become equalities. In particular,  \begin{equation}\label{123}
\|f\|_{L^2(\Omega)} =1= \|f\|_{L^2(\Omega_1)}.\end{equation}
This implies $\Omega=\Omega_1$, i.e., $\Omega$ is a ball. In fact, if  $\Omega \neq \Omega_1$, then  $f= 0$ almost everywhere on the non-empty open set $\Omega \setminus \overline{\Omega_1}$ (of positive Lebesgue measure) by (\ref{123}).
By the holomorphicity we know that $f\equiv 0$ almost everywhere on $\Omega$, which contradicts the fact that $\|f\|_{L^2(\Omega)} =1$.
\medskip

Secondly,  we deal with the relation between the Azukawa indicatrix and the distance function.
Once again, we may suppose $z = 0$.  Since $\mathbb{B}^n(0, \delta(0)) \subset \Omega$, 
\begin{equation}\label{two}
\mathbb{B}^n(0, \delta(0)) = I^A_{\mathbb{B}^n(0, \delta(0))}(0) \subset I_{\Omega}^A(0),
\end{equation}
which gives the inequality \eqref{A}.

 Assume equality in \eqref{A} holds. Then
\begin{equation}\label{six2}
v(I^A_{\Omega}(0) \setminus \mathbb{B}^n(0, \delta(0))) = 0.
\end{equation} We first claim that for all $X\in \partial \mathbb{B}^n(0, \delta(0))$, \begin{equation}\label{six1}
\limsup_{\lambda \to 0} G_0(\lambda X) - \log |\lambda| \leq 0.
\end{equation}
 Indeed, if $\limsup_{\lambda \to 0} G_0(\lambda X) - \log |\lambda| =\epsilon_0>0 $, then for $t>0$ sufficiently close to $1^-$ (say $t> e^{-\epsilon_0} )$,
 \begin{equation*}
     \begin{split}
         \limsup_{\lambda \to 0} G_0\left(\lambda (tX)\right) - \log |\lambda| &= \limsup_{\lambda \to 0} G_0(\lambda tX) - \log |\lambda t| +\log t \\
         &= \left(\limsup_{\lambda t \to 0} G_0\left((\lambda t\right)X) - \log |\lambda t|\right) +\log t\\
         &=\epsilon_0 +\log t >0.
     \end{split}
 \end{equation*} 
This would mean $tX\notin  I_{\Omega}^A(0)$, which contradicts  (\ref{two}). The claim is proved. The same argument also shows that if $X\in I_{\Omega}^A(0)$, then so is $tX$ for all $0\le t\le 1$.

By  (\ref{six1}),  there is a function $f: \partial \mathbb{B}^n(0, \delta(0)) \to [1, \infty)$ such that
$$
\left( I^A_{\Omega}(0) \cup \partial \mathbb{B}^n(0, \delta(0))\right)\setminus \mathbb{B}^n(0, \delta(0)) = \{r\omega: \omega \in \partial \mathbb{B}^n(0, \delta(0)), 1 \leq r \leq f(\omega)\}.
$$
Then by (\ref{six2}),
$$
0 = \int_{\partial \mathbb{B}^n(0, \delta(0))}\int_1^{f(\omega)} r^{2n - 1}dr d\sigma(\omega).
$$
Thus $f(\omega) = 1$ almost everywhere with respect to $\sigma(\partial \mathbb{B}^n(0, \delta(0)))$.
Thus, for almost every $X \in \partial \mathbb{B}^n(0, \delta(0))$, 

\begin{equation} \label{four}
\limsup_{\lambda \to 0} G_0(\lambda X) - \log |\lambda | = 0.
\end{equation}

For fixed  $X \in \partial \mathbb{B}^n(0, \delta(0))$ satisfying \eqref{four}, consider $u: \mathbb{D}(0, 1) \setminus \{0\} \to \mathbb{R}$ by
$$
u(\lambda) = G_0(\lambda X) - \log |\lambda|.
$$
By \eqref{four}, $u$ extends subharmonically to equal 0 at 0.   On the other hand, by the monotonicity of the pluri-complex Green's functions,
\begin{eqnarray*}
u(\lambda) \leq \log \left | {\lambda X \over \delta(0)} \right | - \log |\lambda | = 0.
\end{eqnarray*}
Thus, $u \in SH^{-}(\mathbb{D}(0, 1) \setminus \{0\})$.   By the maximum principle $u \equiv 0$, or equivalently, 
\begin{equation*} 
G_0(\lambda X) \equiv  \log \left | {\lambda X \over \delta(0)} \right |
\end{equation*}
  for almost every $X \in \partial \mathbb{B}^n(0, \delta(0)))$.  Since $G_0 < 0$ on $\Omega$, there is a (possibly empty) $\sigma(\partial \mathbb{B}^n(0, \delta(0)))$-null set $\mathcal{E}$ such that
$$
\partial \mathbb{B}^n(0, \delta(0)) \setminus \mathcal{E} \subset \partial \Omega.
$$
Since $\partial \Omega$ is closed and $\mathbb{B}^n(0, \delta(0)) \subset \Omega$,
$$
\partial \Omega = \partial \mathbb{B}^n(0, \delta(0)), \quad \Omega = \mathbb{B}^n(0, \delta(0)).
$$

\end{proof}

In the case of pseudoconvex domains, the rigidity result concerning the Azukawa indicatrix and the distance function follows immediately from the multi-dimensional Suita conjecture and the first part of our proof of Theorem \ref{delta-c}. 
 Recently, the first author and Wong \cite{DW, DW2} used curvature properties of the Bergman metric to characterize pseudoconvex domains that are biholomorphic to a ball $\mathbb B^n$ possibly less a relatively closed pluripolar set.

When $n = 1$, since $I^A_{\Omega}(z) = \mathbb{D}(0, c_{\beta; \Omega}^{-1}(z))$, Theorem \ref{delta-c} implies the corollary below.  Here we provide an alternative proof using more elementary methods.

\begin{corollary} \label{new}
Let $\Omega$ be a bounded domain  in $\mathbb C$. Then for any $z\in \Omega$,
\begin{equation} \label{cleq}
\delta ^{-1}(z) \ge  c_{\beta}(z).
\end{equation}
Moreover, equality in \eqref{cleq} holds at some $z_0 \in \Omega$ if and only if $\Omega$ is a disk centered at $z_0$.
\end{corollary}

\begin{proof}
For any $z\in \Omega$, write the Green's function as $G_z(w) = \log|z-w| + \rho (w)$. Then $\rho$ is harmonic in $\Omega$ and $\rho(z) = \log c_{\beta}(z)$.
For any positive number $R<\delta(z)$, we have
      $$
      \frac{1}{2\pi} \int_0^{2\pi} G_z(z+R e^{it})dt = \frac{1}{2\pi} \int_0^{2\pi} \log R + \rho ( z+R e^{it})dt.
      $$
By the mean value theorem for harmonic functions,  $\frac{1}{2\pi} \int_0^{2\pi}  \rho ( z+R e^{it})dt = \rho( z) = \log c_{\beta}(z)$. Thus
\begin{equation}\label{cg}
      \frac{1}{2\pi} \int_0^{2\pi} G_z( z+R e^{it})dt  =\log (R c_{\beta}(z)).
      \end{equation}
Recall that for $w \in \bar \Omega  \setminus \{z\}$,
$
G_z( w) \le  0.$
This implies from (\ref{cg}) that
$$
  c_{\beta}(z) \le  R^{-1}.
$$
Letting $ R \to \delta(z)^-$, we obtain (\ref{cleq}).

Assume $c_{\beta}(z_0) = \delta^{-1}(z_0)$ at some point $z_0\in \Omega$. Then by (\ref{cg}) and the nonpositivity of $G_{z_0}$, $G_{z_0}( w)\rightarrow 0^-$
for all $w\in \partial \mathbb D(z_0, R)$ as $R\rightarrow \delta(z_0)^-$.  The continuity and nonpositivity of $G_{z_0}$ in $\Omega$ further concludes that this can only happen when $\partial \Omega$ coincides with $\partial \mathbb D(z_0, \delta(z_0))$.

\end{proof}

\begin{remark}\label{Final remark}
When $\Omega$ is a bounded domain in $\mathbb C$, Theorem \ref{delta-c} reduces to  $$\frac{1 }{ \delta^{2}(z)} \ge \pi K(z) $$
for all $z\in \Omega$ and equality holds at some $z_0\in \Omega$ if and only if $\Omega$ is a disk centered at $z_0$. Combining this with  
 the previous inequality chains considered in the paper, we know that
there exists some $z_0\in \Omega$ (and some $t_0 \le 0$) such that $$
 \frac{1}{\delta^2(z_0)}  = %\pi K(z_0),  \quad c_{\beta}^2(z_0),\quad 
 c_B^2(z_0)\quad \text{or} \quad \frac{\pi e^{2{t_0}} }{v{(\{G( \cdot, z_0) < t_0\})}}  \quad \text{or} \quad {v(\mathbb{C}_{\infty} \setminus j_{z_0}(\Omega)) \over \pi}
$$
if and only if $\Omega$ is a disk centered at $z_0$.

\end{remark}

\begin{appendices}   
\appendix
 \bigskip
 \noindent{\LARGE{\bf Appendices}}

\section{Rigidity Phenomenon Between $c_B$ and $c_{\beta}$} \label{AppendixA}

The aim of this appendix is to prove the following rigidity theorem characterizing the equality conditions for $c_B \leq c_{\beta}$.

\begin{theorem}
\label{main2}

For an open hyperbolic Riemann surface $X$, $c_{\beta}(z_0)=c_B(z_0)$ for some $z_0 \in X$
if and only if  
$X$ is biholomorphic to the unit disk $\mathbb D$ possibly less a relatively closed polar subset $P$;
in this case, 
the extremal functions of the analytic capacity $c_B(z_0)$ equal $ \varphi \circ f $, where $f: X \to \mathbb D \setminus P$ is a  biholomorphism and  $\varphi \in \text{Aut}(\mathbb D)$ such that  $ \varphi \circ f(z_0) =0$.

\end{theorem}

Given two Riemann surfaces $X$ and $Y$ with $X$ admitting a Green's function,  Minda in \cite[Theorem 3]{Min} proved that  if $f:X \to Y$ is holomorphic, then 
$$
f^*(c_{\beta; Y}(w)|dw|) \leq c_{\beta; X}(\zeta)|d\zeta|.
$$ 
Moreover, if equality holds at a single point, then $f$ is biholomorphic onto its image and $Y\setminus f(X)$ is a closed polar set. Although Minda's approach does not mention the analytic capacity, it can be used to deduce Theorem \ref{main2} as follows.  For any $f\in \hbox{Hol}(X, \mathbb D)$, 
$$
   |f'(z)||dz|\leq {|f^{\prime}(z)| \over 1 - |f(z)|^2} |dz|  = f^{*}(c_{\beta; \mathbb{D}}(z))|dz|\leq c_{\beta}(z)|dz|.
$$

If we suppose $c_{\beta}(z_0) = c_B(z_0)$ and set $f = f_0$ to be an extremal function for $c_{B}(z_0)$, then all inequalities above become equalities at $z_0$. By Minda's result, $f_0$ is a biholomorphism onto $\mathbb{D} \setminus P$ where $P$ is a relatively closed polar set.  Thus, Theorem \ref{main2} is proved. 
 The key ingredient of Minda's proof is the ``sharp form of the Lindel\"of principle'' for the Green's function. In the following, we shall reprove  Theorem \ref{main2} without resorting to this principle. 
 
  \medskip 
  
 Let $f$ be a holomorphic function on an open Riemann surface $X$ such that $f(X) \subset \mathbb{C}$ admits a Green's function. Then the Green's function satisfies a {\bf subordination property}:
\begin{equation} \label{subordination}
G_{X}(z, z^{\prime}) \geq G_{f(X)}(f(z), f(z^{\prime}))
\end{equation}
for all $(z, z^{\prime})\in X\times X$. Moreover, if there exists  some $(z_0, z_0^{\prime})\in X\times X$ with $ z_0\ne z_0'$ such that equality in (\ref{subordination}) holds, then
\begin{equation} \label{subordination1}
G_{X}(z, z^{\prime}) \equiv G_{f(X)}(f(z), f(z^{\prime}))
\end{equation}
for all $(z, z^{\prime})\in X\times X$ and $f$ is injective. 
See for instance \cite[Theorem 4.4.4]{R95} for the proof for  planar domains. The cases for Riemann surface can be proved similarly. A consequence of the property is the following rigidity property Lemma \ref{main} of the Green's function. We note that strengthened forms of the subordination property and Lemma \ref{main}  were proved in \cite[Theorem 1]{Min} by Minda using the Lindel\"of principle for the Green's function.

\begin{lemma}[Rigidity lemma of the Green's function] \label{main}

On an open Riemann surface $X$, the Green's function with a pole $z_0 \in X$ is
 $$
G(z, z_0)= \log|f(z)|, \quad z\in X
 $$
for some holomorphic function $f$ on $X$ if and only if $f$ is a biholomorphism from $X$ to the unit disk possibly less a relatively closed polar subset such that $f(z_0) =0$.

\end{lemma}
\begin{proof}
Since $G_X(z, z_0) < 0$ for $z \in X$, $f(X) \subset  \mathbb D$. The image $f(X)$ admits a Green's function because $f(X)$ is bounded.  Also, observe that $f(z_0)=0$.
Applying the subordination property \eqref{subordination} to $f$ and the identity map,  we get 
$$
\log|f(z)| = G_{X}(z, z_0) \geq G_{f(X)}(f(z), 0) \geq  G_{\mathbb D}(f(z), 0) = \log |f(z)|.
$$
By the subordination property \eqref{subordination1}, $f$ is injective and $G_{f(X)}(\zeta, 0) = \log |\zeta|$ for $\zeta\in f(X)$.  Let $\eta \in \partial f(X) \cap \mathbb D$ and $\zeta_n(  \in f(X) ) \to \eta$. Since
$$
\lim_{n \to \infty}G_{f(X)}(\zeta_n, 0)= \lim_{n \to \infty} \log  |\zeta_n|= \log |\eta|  < 0,
$$
$\eta$ is an irregular boundary point.  By Kellogg's Theorem, cf. \cite[Theorem 4.2.5]{R95}, $P = \partial f(X) \cap \mathbb D$ is a relatively closed polar set  in $\mathbb D$.  Suppose $w_0 \in \mathbb{D} \setminus \overline{f(X)}$.  Then for some $\epsilon > 0$, $f(X) \subset \mathbb{D} \setminus \overline{\mathbb{D}(w_0, \epsilon)}$.  Let $k$ be the harmonic function defined on the latter set with Dirichlet boundary data
\[
k(z) =
\begin{cases}
      0, & z \in \partial \mathbb{D} \\
      -\log (|w_0| + \epsilon), & z \in \partial\mathbb{D}(w_0, \epsilon).
   \end{cases}
\]
Since $G_{f(X)}(z, 0) \leq G_{f(X)}(z, 0) + k(z)$ and $G_{f(X)}(z, 0) + k(z)$ is in the defining set of \eqref{Other Green's Function} for the domain $f(X)$, we have arrived at a contradiction unless $\mathbb{D} \setminus \overline{f(X)}$ is empty. Thus, $\mathbb{D} = f(X) \sqcup P,$ where $\sqcup$ denotes the disjoint union. The proof of the theorem is complete.

\end{proof}

In \cite[Theorem 3.1, p. 1196]{GZ3} of the equality part of the Suita conjecture, Guan and Zhou showed that if $\pi K(z_0) = c_{\beta}^2(z_0)$, then by their optimal $L^2$ extension theorem, $\exp G(z, z_0) = |g(z)|$, for
some holomorphic function $g$ on $\Omega$. 
By proving further  $c_{\beta}^2(z_0) = c_B^2(z_0)$ if $\exp G(z, z_0) = |g(z)|$, they were able to apply  Suita's Theorem (Theorem \ref{Suita Theorem}) to yield that $X$ is biholomorphic to a disk less a relatively closed polar set.

In fact, 
Lemma \ref{main} says that $|g(z)| = \exp G(z, z_0)$ if and only if
the surface is biholomorphic to a disk less a relatively closed polar set. Consequently,  one does not need to involve the analytic capacity or Suita's Theorem   
to prove the equality part of the Suita conjecture as in \cite{GZ3}. See also \cite{D} for related work.  

\medskip

Next,  by relying more explicitly on the analytic capacity, we shall give a proof of Theorem \ref{main2} by making use of Lemma \ref{main} and following an idea of Guan and Zhou in \cite[Lemma 4.25]{GZ3}.

\begin{proof} [{\bf Proof of Theorem \ref{main2}.} ] 
If $X$ is biholomorphic to a disk less a relatively closed polar set, then $c_{\beta} \equiv c_B$, since the polar part is negligible.

\medskip

Conversely, write $G_{z_0}(z)$ for $G(z, z_0)$, and let $u(z): = \log (|f_0(z)|)$ for $z\in X $, where $f_0$ is an extremal function of the analytic capacity at $z_0$. Since   $|f_0|< 1$ on $X $, $u\in SH^{-}( X)$. By definition of the Green's function, we further see that $u-G_{z_0}\in SH^{-}( X)$.
We will show that  $u-G_{z_0}|_{z=z_0} = 0$. In the local coordinate $w(z)$ of $X$ near $z_0$,   $\lim_{z\rightarrow z_0} \log |f_0(z)|-\log |w(z)|=\log |\frac{df_0}{dw}(z_0)| = \log c_B(z_0)$ by definition of $f_0$. Thus for $z\in X$ near $z_0$,
$$
        u(z)-G_{z_0}(z) = (\log (|f_0(z)|) - \log |w(z)|) - (G_{z_0}(z)- \log |w(z)|) \rightarrow  \log c_B(z_0) - \log c_{\beta}(z_0) = 0.
 $$
As a consequence of  the maximum principle of subharmonic functions, we have
$$
G_{z_0}(z)= \log (|f_0(z)|).
$$
By Lemma \ref{main}, $f_0$ is a biholomorphism from $X$ to the unit disk less a relatively closed polar set, with $f_0(z_0)=0$.

\medskip

To complete the second part of the theorem, 
let $f: X \to \mathbb D \setminus P$ be any biholomorphism, and let
$
\varphi (z):= \frac{z- f(z_0)}{ 1- \overline {f(z_0)} z} \in \text{Aut} (\mathbb D).
$
 Then, $F:= \varphi \circ f: X \to \mathbb D \setminus \varphi (P) $ is also a biholomorphism such that  $F(z_0) =0$, and $G_{z_0} =  \log |F|$ by Lemma \ref{main}.  We will show that $F$ is an extremal function for $c_B(z_0)$. For any $h\in \hbox{Hol}(X, \mathbb D) $ with $h(z_0) =0$ and $h^{\prime}(z_0) \neq 0$, we have that $\log |h|\in SH^- (X)$. By  the definition of the Green's function, 
  $\log |h|\le G_{z_0} = \log |F|$ and so $|h'(z_0)|\le |F'(z_0)|$. Therefore,
$$
        c_B(z_0) = |F'(z_0)|,
$$
which means $F$ is extremal.

\end{proof}

It is known (see \cite[VII.5H]{SO}) that  if $X$ is a regular region of connectivity greater than or equal to 2, then
\begin{equation} \label{>}
c_{\beta}(z) > c_B(z), \quad \hbox{for all } z\in X.
\end{equation}
Theorem \ref{main2} says that \eqref{>} in fact holds for any hyperbolic Riemann surface $X$ which is not biholomorphic to a disk possibly less a relatively closed polar subset.

\section{Stability of the Szeg\"o Kernel on Lipschitz Domains}\label{AppendixC}

The Szeg\"o projection and its kernel on planar domains with Lipschitz boundaries  were studied extensively by Lanzani \cite{L99}. One of the key steps in extending results 
from smoothly bounded domains to those with Lipschitz boundaries is by approximating the bounded Lipschitz domain $\Omega$ by smoothly bounded subdomains in the sense of Ne\v{c}as (see \cite[Theorem 2.3]{L99}).

\begin{theorem}{\normalfont\cite{L99}}\label{approximation in the sense of Necas}
Let $\Omega \subset \mathbb{C}$ be a bounded domain with Lipschitz boundary.  Then there exists a sequence of subdomains $\{\Omega_j\}_{j=1}^{\infty}$ with $C^{\infty}$ smooth boundaries such that
\begin{enumerate}
    \item [$1.$] There is a sequence of Lipschitz homeomorphisms $\Lambda_j: \partial \Omega \to \partial \Omega_j$ such that $\Lambda_j(P) \in \Gamma(P)$ (see Definition \ref{Approach})  and $\lim_{j\to \infty} \Lambda_j(P) = P$.
    \item [$2.$] There exists functions $\omega_j: \partial \Omega \to (0, \infty)$ that are uniformly bounded away from 0 and $\infty$, converge to 1 almost everywhere and 
  $$
    \int_{\partial \Omega} h(\Lambda_j(w))\omega_j(w) d\sigma(w) = \int_{\partial \Omega_j}h(\zeta)d\sigma_j(\zeta)
 $$
 for any $h \in L^1(\partial\Omega_j)$.
    \item [$3.$] The unit tangent vector $T_j$ for $\partial \Omega_j$ and the unit tangent vector $T$ for $\partial \Omega$ satisfy that $T_j(\Lambda_j(\cdot)) \to T(\cdot)$ almost everywhere on $\partial \Omega$.  
\end{enumerate}
\end{theorem}

If $\Omega$ has  $C^{\infty}$ boundary with a defining function $\rho$ and Szeg\"o kernel $S$, define the sublevel sets of the defining function for small $\epsilon>0$ 
$$
\Omega_{\epsilon} = \{z \in \Omega: \rho(z) < -\epsilon\},
$$
and consider the Szeg\"o kernels $S_{\epsilon}(\cdot, \cdot)$ of these respective domains.  In the proof of Theorem 2.1 of \cite{Bo87}, Boas showed that for $a, z \in  \Omega_{\epsilon}$,
\begin{equation*}
\lim_{\epsilon \to 0^{+}} S_{\epsilon}(z, a) = S(z, a).
\end{equation*}

The purpose of  this section is to extend this stability result of the Szeg\"o kernel to Lipschitz domains in $\mathbb{C}$ with respect to the Ne\v{c}as approximation. We begin by giving  notations along the lines of \cite{L99}.

\begin{definition}\label{Approach}
 \begin{enumerate}
 
\item [$1.$] For $\lambda > 0$, $P \in \partial \Omega$, the non-tangential approach region to P is
$$
\Gamma(P) = \{\zeta : |\zeta - P| \leq (1 + \lambda)dist(\zeta, \partial\Omega)\}
$$

\item [$2.$] For a function $f$ defined on a Lipschitz domain, the non-tangential limit $f^{+}$ and non-tangential maximal function $f^{*}$, if they exist, are defined respectively as
$$
f^{*}(P) = \sup_{w \in \Gamma(P)} |f(w)|, \quad f^{+}(P) = \lim_{\substack{w \to P \\ w \in \Gamma(P)}} f(w), \quad P \in \partial \Omega.
$$

\item [$3.$]

The Hardy space of a bounded domain $\Omega$ with Lipschitz boundary is
$$
H^2(\partial \Omega) = \{ f^{+} : f \in \mathcal{O}(\Omega), f^{*} \in L^2(\partial \Omega)\}.
$$
 \end{enumerate}
\end{definition}

A useful characterization of the Szeg\"o kernel for our purposes is
\begin{equation}\label{Characterization of Szego}
S(z, a) = {f(z) \over \|f^{+}\|^2_{\partial \Omega}},
\end{equation}
where $f$ is the unique function with minimal $L^2(\partial \Omega)$-norm among all functions in $H^2(\partial\Omega)$ with $f(a) = 1$.   Using Theorem \ref{approximation in the sense of Necas} we can adapt the proof of Boas  \cite[Theorem 2.1]{Bo87}.

\begin{proposition}\label{Our Szego Lemma}
Let $\Omega$ be a bounded domain with Lipschitz boundary and $S(\cdot, \cdot)$ be its Szeg\"o kernel.  Let $\{\Omega_j\}_{j=1}^{\infty}$ be a sequence of subdomains as in Theorem \ref{approximation in the sense of Necas}.  Denote by $\{S_j(\cdot, \cdot)\}$ the corresponding Szeg\"o kernels for these domains.  Then for each $a, z \in \Omega$,
$$
\lim_{j \to \infty} S_j(z, a) \to S(z, a).
$$
\end{proposition}

\begin{proof}
Let $\{f_j\}$ and $f$ be the extremal functions for the domains $\{\Omega_j\}_{j=1}^{\infty}$ and $\Omega$ as in \eqref{Characterization of Szego}. Since 
$
|f \circ \Lambda_j|^2 \omega_j \leq C|f^*|^2 \in L^1(\partial \Omega),
$
\begin{equation}\label{Lipschitz one}
\|f^{+}_j\|^2_{\partial \Omega_j} \leq \|f\|^2_{\partial \Omega_j} = \|(f \circ \Lambda_j)\omega_j^{{1 \over 2}}\|^2_{\partial \Omega} \leq M < \infty
\end{equation}
for some constant $M$, and by the dominated convergence theorem
\begin{equation}\label{follows by DCT}
\lim_{j \to \infty} \|(f \circ \Lambda_j)\omega_j^{{1 \over 2}}\|^2_{\partial \Omega} = \|f^+\|^2_{\partial \Omega}.
\end{equation}
By \eqref{Lipschitz one}, on each $\Omega_k$, $\{f_j\}_{j \geq k}$ is a normal family. After passing to a subsequence, there exists an $F \in \mathcal{O}(\Omega)$ such that $f_j \to F$ uniformly on compact subsets. Necessarily this implies $F(a) = 1$. 

By \eqref{Lipschitz one}, after passing to a subsequence $(f_j^{+}\circ \Lambda_j)\omega_j^{{1/2}}$ converges weakly to $f_{\infty}$ in $L^2(\partial\Omega)$.  Let $\mathcal{C}$ denote the Cauchy transform. 
We claim that 
\begin{equation*}
    F(z) = \mathcal C(f_\infty)(z) \left(:= {1 \over 2\pi i} \int_{\partial \Omega} {f_{\infty}(w) \over w - z} dw \in  H^2(\partial\Omega)\right).
\end{equation*}
In fact for any fixed $z\in \Omega_j$, 
\begin{equation*}
    \begin{split}
        |f_j(z) - \mathcal C(f_\infty)(z)|= &\left| \int_{\partial\Omega}{f_j^{+} ( \Lambda_j(w)) T_j(\Lambda_j(w))\omega_j(w) \over \Lambda_j(w) - z}d\sigma -\int_{\partial\Omega}\frac{f_\infty(w) T(w)}{w-z}d\sigma\right|\\
       \le & \int_{\partial\Omega}|f_j^{+} ( \Lambda_j(w))\omega_j^\frac{1}{2}(w)|\left| {T_j(\Lambda_j(w))\omega_j^\frac{1}{2}(w) \over \Lambda_j(w) - z} - \frac{T(w)}{w-z}\right|d\sigma\\
       & \quad + \left|\int_{\partial\Omega} \left(f_j^{+} ( \Lambda_j(w))\omega_j^\frac{1}{2}(w) - f_\infty(w)\right)\frac{T(w)}{w-z}d\sigma\right|: =A+B.
    \end{split}
\end{equation*}
Since $\frac{T(w)}{w-z}\in L^\infty(\partial \Omega)$ and $(f_j^{+} \circ \Lambda_j)\omega_j^\frac{1}{2}$ converges weakly to $f_\infty$ in $L^2(\partial\Omega)$, we have $B\rightarrow 0$. For $A$, notice that by Theorem \ref{approximation in the sense of Necas} and the dominated convergence theorem, $ {T_j(\Lambda_j(w))\omega_j^\frac{1}{2}(w) \over \Lambda_j(w) - z}  $ converges to $ \frac{T(w)}{w-z} $ in $L^2(\partial\Omega)$ norm. Making use of H\"older inequality and  the uniform boundedness of $\|(f_j^{+}\circ \Lambda_j)\omega_j^{1/2}\|_{\partial \Omega}$,  we get
$$
A\le \|(f_j^{+}\circ \Lambda_j)\omega_j^{\frac12}\|_{\partial \Omega}\left\| {T_j(\Lambda_j(w))\omega_j^\frac{1}{2}(w) \over \Lambda_j(w) - z} - \frac{T(w)}{w-z}\right\|_{\partial \Omega}\rightarrow 0.
$$
Hence $f_j \to \mathcal{C}f_{\infty}$ pointwisely, and so $F = \mathcal{C}(f_{\infty})$.  

\medskip

Using the fact that $\{\overline{TG}: G \in H^2(\partial\Omega)\} = H^2(\partial\Omega)^{\perp}$, we will show that $f_{\infty} \in H^2(\partial\Omega)$, cf. \cite[Theorem 5.1]{L99}, \cite[Theorem 4.3]{Be16}. Since
$\langle f_j^{+}, \overline{T_j G} \rangle_{\partial \Omega_j} = 0$,
\begin{eqnarray*}
\lefteqn{|\langle f_{\infty}, \overline{TG} \rangle|} 
\\
&=& |\langle f_{\infty} - (f_j^{+} \circ \Lambda_j) \omega_j^{1 \over 2}, \overline{TG} \rangle_{\partial\Omega} 
\\
&& \quad + \int_{\partial\Omega} f_j^{+}(\Lambda_j)(w))\omega_j^{1 \over 2}(w)T(w)G(w) - (f_j^{+} (\Lambda_j)(w))T_j(\Lambda_j(w))G(\Lambda_j(w))\omega_j(w) d\sigma(w)|
\\
&\leq & |\langle f_{\infty} - (f_j^{+} \circ \Lambda_j)\omega_j^{1 \over 2}, \overline{TG} \rangle_{\partial\Omega}| + M^{1 \over 2}\Big{(}\int_{\partial\Omega} |T(w)G(w) - T_j(\Lambda_j(w))G(\Lambda_j(w))\omega_j^{1 \over 2}(w)|^2 d\sigma(w)\Big{)}^{1 \over 2}.
\end{eqnarray*}

The first term above approaches 0 as $j$ approaches $\infty$ by weak-convergence.  The integrand in the second term is dominated by $C|G^{*}|^2$ for some constant $C$.  By the dominated convergence theorem, the integral approaches 0 as well. Thus,
$
f_{\infty} \in  H^2(\partial\Omega).
$
Since the Cauchy transform is a projection onto $H^2(\partial\Omega)$, $f_{\infty} = F^{+}$ on $\partial\Omega.$ In particular $(f_j^+ \circ \Lambda_j) {\omega_j }^\frac{1}{2}  $  converges weakly to $F^+$ in $L^2(\partial\Omega)$ as well.
Hence

\begin{eqnarray}
        \|F^+\|^2_{\partial \Omega}&\le& \liminf_{j\rightarrow \infty}  \|(f_j \circ \Lambda_j)\omega_j^{{1 \over 2}}\|^2_{\partial \Omega} = \liminf_{j\rightarrow \infty} \|f_j^+\|^2_{\partial \Omega_j} \nonumber \\
        &\le& \limsup_{j\rightarrow \infty} \|f_j^+\|^2_{\partial \Omega_j} \le \limsup_{j\rightarrow \infty}\|f\|^2_{\partial \Omega_j} = \|f^+\|^2_{\partial \Omega}. \label{in}
\end{eqnarray}
Here the first inequality uses the weakly sequentially lower semicontinuity for the $L^2(\partial\Omega)$ norm, the third  inequality applies  (\ref{Lipschitz one}) and the last equality uses \eqref{follows by DCT}.  Since $f$ is unique, $F = f$ on $\Omega$, and all inequalities in (\ref{in}) become equalities. In particular, $f_j$ converges pointwisely to $f$ on $\Omega$ and  $\|f_j^+\|_{\partial \Omega_j}\rightarrow \|f^+\|_{\partial \Omega}$ by \eqref{in}.
By  \eqref{Characterization of Szego},
$
S_j(z, a) \to S(z, a). 
$

\end{proof}

 \end{appendices}

 \subsection*{Funding}

{\fontsize{11.5}{10}\selectfont

The research of the first and second author is supported by an AMS-Simons travel grant. The research of the third author is partially supported by NSF grant DMS-1501024.

}

  \subsection*{Acknowledgements}
{\fontsize{11.5}{10}\selectfont Some of the results in this paper are a part of the second-named author's Ph.D. dissertation \cite{Tr} at the University of California, Irvine.  He would like to thank his thesis advisor Song-Ying Li for the support throughout the years, and Harold Boas and Emil Straube for useful conversations and encouragement during the writing process.   The authors thank Boas \cite{Bo22} for explaining the relevance of 
 the shears   to our 
 paper.  
 Finally, we thank the anonymous referee for valuable comments, and Marc Benoit for spotting a few typos in an earlier draft.
 
 }

\bibliographystyle{alphaspecial}

\fontsize{11}{9}\selectfont

\vspace{0.5cm}

\noindent dong@uconn.edu,

\vspace{0.2 cm}

\noindent Department of Mathematics, University of Connecticut, Storrs, CT 06269-1009, USA

\vspace{0.4cm}

\noindent (formerly) jtreuer@uci.edu, (now) jtreuer@tamu.edu;

 \vspace{0.2 cm}

\noindent Department of Mathematics, Texas A\&M University, College Station, TX 77843-3368, USA

\vspace{0.4cm}

\noindent zhangyu@pfw.edu,

\vspace{0.2 cm}

\noindent Department of Mathematical Sciences, Purdue University Fort Wayne, Fort Wayne, IN 46805-1499, USA

\end{document}